\documentclass[11pt,reqno]{amsart}
\usepackage[T1]{fontenc}
\usepackage{graphicx,   amssymb, xcolor}
\usepackage{cite}
\usepackage{fixltx2e}
\usepackage{hyperref}

\usepackage{color}
\definecolor{MyLinkColor}{rgb}{0,0,0.4}
\newcommand{\vdiv}{\mathop{\rm div}}
\newcommand{\R}{{\mathbb R}}
\newcommand{\RRM}{{\mathbb R}}
\newcommand{\TTM}{{\mathbb T}}

\newcommand{\bA}{{\mathbb A}}
\newcommand{\bB}{\mathbb{B}}
\newcommand{\bD}{\mathbb{D}}

\newcommand{\kB}{\mathcal{B}}

\newcommand{\N}{{\mathbb N}}
\newcommand{\clm}{{\mathcal M}}

\newcommand{\kL}{\mathcal{L}}

\newcommand{\cP}{\mathcal{P}}
\newcommand{\cU}{\mathcal{U}}
\newcommand{\clu}{\mathcal{U}}
\newcommand{\clp}{\mathcal{P}}
\newcommand{\clw}{\mathcal{W}}
\newcommand{\clq}{\mathcal{Q}}

\newcommand{\clz}{\mathcal{Z}}
\newcommand{\oo}{\overline\omega}
\newcommand{\re}{\mathop{\rm Re}\nolimits}

\newcommand{\PV}{\mathop{\rm PV}\nolimits}
\newcommand{\ov}{\overline}
\newcommand{\p}{\partial}

\newcommand{\e}{\varepsilon}

\newcommand{\0}{\Omega}
\newcommand{\G}{\Gamma}

\newcommand{\supp}{\mathop{\rm supp}\nolimits}
\newcommand{\be}{\begin{equation}}
\newcommand{\ee}{\end{equation}}
\newcommand{\rot}{\mathop{\rm rot}\nolimits}
\newtheorem{thm}{Theorem}[section]
\newtheorem{prop}[thm]{Proposition}
\newtheorem{lemma}[thm]{Lemma}
\newtheorem{cor}[thm]{Corollary}
\theoremstyle{remark} 
\newtheorem{rem}[thm]{Remark}

\numberwithin{equation}{section}
\begin{document}

\title[Two-phase Stokes flow]{Two-phase Stokes flow by capillarity in the plane:\\ the case of different viscosities}

\author{Bogdan--Vasile Matioc}
\address{Fakult\"at f\"ur Mathematik, Universit\"at Regensburg,   93040 Regensburg, Deutschland.}
\email{bogdan.matioc@ur.de}
  
\author{Georg Prokert}
\address{ Faculty of Mathematics and Computer Science, Technical
University Eindhoven,   The Netherlands.}
\email{g.prokert@tue.nl}

%%% NEED TO ADD MSC %%%

\begin{abstract}
We study the two-phase Stokes flow driven by surface tension for two fluids of different viscosities, separated by an asymptotically flat interface 
 representable as graph of a differentiable function.
 The flow is assumed to be two-dimensional with the fluids filling the entire space.
 We prove well-posedness and parabolic smoothing in Sobolev spaces up to critical regularity. The main technical tools are an analysis of nonlinear singular integral operators arising from the hydrodynamic single and double layer potential, spectral results on the corresponding integral operators,  and abstract results on nonlinear parabolic evolution equations. 
\end{abstract}

\maketitle

\section{Introduction}

 In the context of boundary value problems involving elliptic constant-coefficient PDE's like the Laplace equation or the Stokes system, it is often natural to consider two-phase problems in unbounded domains, where the same equation has to be solved on both sides of the  boundary, and the boundary conditions typically are of ``transmission'' type, i.e. they relate limits of the solutions from both sides. The method of layer potentials is a classical technique which is intrinsically suited to such settings. Typically, this method reduces the boundary value problem to a linear, singular integral equation (or system of such equations) on the boundary of the domain, on the basis of well-known jump relations for these potentials across the boundary.

The first applications of layer potentials in the analysis of moving boundary problems of the type described above are from the 1980s, 
for problems of Hele-Shaw or Muskat type~\cite{DR84} (see also the recent surveys~\cite{G17, GL20} on further developments) as well as for Stokes flow problems~\cite{BaDu98}. In these applications, the interfaces are represented as graphs of a time dependent function~${[f\mapsto f(t)]}$, with~${f(t)\in{\rm C}(\R)},$ for which an evolution equation can be derived. 
This equation involves singular integral operators originating from the layer potential, depending nonlinearly and nonlocally on~$f(t)$. 
However, in suitable geometries this nonlinearity can be described rather explicitly, and technicalities resulting from transforming the problem to a fixed reference domain can be avoided. More precisely, the operators determining the evolution belong to a class discussed in Section \ref{Sec:10} below, and results are available concerning mapping properties, smoothness, localization etc. of the operators in this class.

  After reducing the moving boundary problem to an evolution equation for $f$, this equation has to be analyzed. 
   Initially, various approaches  have been used that necessitated rather restrictive assumptions on the initial data.
 Recently, however, more general, in some sense optimal existence, uniqueness, and smoothness results have been obtained. One of the crucial tools for this has been the meanwhile well-developed and versatile abstract theory of nonlinear parabolic evolution equations, cf. \cite{L95, Am93, PS16}.

This paper discusses, along the lines sketched above, the moving boundary problem of two-phase Stokes flow in full 2D space driven by surface tension forces on the interface between the two phases. 
More precisely, we seek a moving interface $[t\mapsto\Gamma(t)]$ between two liquid phases~$\Omega^\pm(t)$, and corresponding functions 
\begin{align*}
    v^\pm(t):\Omega^\pm(t)\longrightarrow\RRM^2\qquad\text{and}\qquad p^\pm(t):\Omega^\pm(t)\longrightarrow\RRM,
\end{align*}
representing the velocity and pressure fields in $\Omega^\pm(t)$, respectively, such that the following equations are satisfied:
\begin{subequations}\label{STOKES}
\begin{equation}\label{probint}
\left.
\begin{array}{rclll}
\mu^\pm\Delta v^\pm-\nabla p^\pm&=&0&\mbox{in $\Omega^\pm(t)$,}\\
\vdiv v^\pm&=&0&\mbox{in $\Omega^\pm(t)$,}\\{}
[v]&=&0&\mbox{on $\Gamma(t)$,}\\{}
[T_\mu(v, p)]\tilde\nu&=&-\sigma\tilde\kappa\tilde\nu&\mbox{on $\Gamma(t)$,}\\
(v^\pm, p^\pm)(x)&\to&0&\mbox{for $|x|\to\infty$, $x\in\Omega^\pm(t)$}\\
V_n&=& v^\pm\cdot\tilde \nu&\mbox{on $\Gamma(t)$.}
\end{array}\right\}
\end{equation}
Here $\tilde\nu$ is the unit exterior normal to~$\p\Omega^-(t)$ and $\tilde\kappa$ denotes the curvature of the interface.  
Moreover, $T_\mu(v,p)=(T_{\mu,ij}(v,p))_{1\leq i,\, j\leq 2}$ denotes the stress tensor  that is given by
\begin{align}\label{defT}
T_{\mu,ij}(v, p):=- p\delta_{ij}+\mu(\partial_iv_j+\partial_j v_i),
\end{align}
 and $[v]$ (respectively $[T_\mu(v, p)]$) is the jump of the velocity (respectively stress tensor) across the moving interface, see \eqref{defjump}  below.
 The positive constants~$\mu^\pm$ and~$\sigma$  denote the viscosity of the liquids in the two phases and the surface tension coefficient of the interface, respectively. 
 We assume that 
 $$ \Gamma(t)=\partial\Omega^+(t)=\p\Omega^-(t), \qquad\Omega^+(t)\cup\Omega^-(t)\cup \Gamma(t)=\mathbb{R}^2, \qquad
 \Gamma(t)=\mathop{\rm graph} f(t)$$
 so that $\Gamma(t)$ is a graph over  the real line. 
 Equation \eqref{probint}$_6$  determines the motion of the interface by prescribing  its normal velocity $V_n$ as coinciding with the
  normal component of the velocity at $\Gamma(t)$, i.e. the interface is transported by the liquid flow. 
 The interface $\Gamma(t)$ is assumed to be known at time $t=0$:
 \begin{align}\label{IC}
 f(0)=f_0.
 \end{align}
 \end{subequations}
 
 In the previous paper \cite{MP2021}, the authors considered Problem \eqref{probint} in the case of equal viscosities $\mu^\pm=\mu$.
  In that case, the solution to the fixed-time problem \eqref{probint}$_1$--\eqref{probint}$_5$ can be directly represented as a 
  hydrodynamic single-layer potential \cite{Lad63} with density $-\sigma\tilde\kappa\tilde\nu$, and the 
  resulting evolution equation represents the time derivative of $f$ as a nonlinear singular integral operator acting on $f$. 
 
 If $\mu^+\neq\mu^-$ this is not feasible. Instead, we first transform the unknowns such that the same equation holds in both phases, introducing thereby 
 a jump across the interface for the transformed velocity field.  
 In Proposition \ref{P:STOsol}, we show that the corresponding fixed-time Stokes problem is uniquely solvable, and we represent the solution
 by  a sum of a hydrodynamic single layer and a double layer potential.
 While the single layer potential is generated by the same density as in the case of equal viscosities, the density $\beta$ for the double layer potential is
   found from solving a linear, singular integral equation of the second kind,  cf. \eqref{TOSO}.
   As~${\Gamma(t)}$ is unbounded we cannot rely on compactness arguments to show the solvability of this equation. 
  Instead, we modify arguments from \cite{FKV88,CP09} to obtain the necessary information on the spectrum of the corresponding integral operator via a Rellich identity.
   Moreover, we also rely on  a further Rellich identity used in \cite{MBV18} in the study of the Muskat problem.
   
 The solution to the fixed-time problem is then used in the formulation of an evolution equation for $f$, (cf. \eqref{evol0}, \eqref{NNEP}, \eqref{PHI})
   \[\frac{df}{dt}(t)=\Phi(f(t)),\quad t\geq 0,\quad f(0)=f_0,\]
   whose investigation will yield the following main result. 
 Here and further,~${H^s(\R):=W^{s}_2(\R)}$ denotes the usual Sobolev spaces of integer or noninteger order.
 
\begin{thm}\label{MT1} Let  $s\in(3/2,2) $ be given.
Then, the following  statements hold true:
\begin{itemize}
\item[(i)]  {\em (Well-posedness)}  Given $f_0\in H^{s}(\mathbb{R})$, there exists a unique maximal solution~$(f,v^\pm,p^\pm)$   to \eqref{STOKES} such that
\begin{itemize}
\item[$\bullet$] $f=f(\cdot;f_0)\in {\rm C}([0,T_+),  H^{s}(\mathbb{R}))\cap {\rm C}^1([0,T_+), H^{s-1}(\mathbb{R})),$\\[-2ex]
\item[$\bullet$] $v^\pm(t)\in {\rm C}^2(\Omega^\pm(t))\cap {\rm C}^1(\overline{\Omega^\pm(t)})$, $p^\pm(t)\in {\rm C}^1(\Omega^\pm(t))\cap {\rm C}(\overline{\Omega^\pm(t)})$ for all ${t\in(0,T_+)}$,\\[-2ex]
\item[$\bullet$] $ v(t)^\pm|_{\G(t)}\circ\Xi(f(t))\in H^2(\R)^2$ for all $t\in(0,T_+)$,\\[-2ex]
\end{itemize}
where $T_+=T_+(f_0)\in (0,\infty]$ and $\Xi(f(t))(\xi):=(\xi, f(t)(\xi))$, $\xi\in\R.$\\ 
 Moreover, the set $$\clm:=\{(t,f_0)\,|\,f_0\in H^s(\RRM),\,0< t<T_+(f_0)\}$$ is open in 
$(0,\infty)\times H^s(\R)$, and $[(t,f_0)\mapsto f(t;f_0)]$  is a semiflow on $H^s(\R)$ which is smooth in $\clm$.\\[-2.2ex]
\item[(ii)]  {\em (Parabolic smoothing)} \\[-2ex]
\begin{itemize}
\item[(iia)]  The map $[(t,\xi)\mapsto  f(t)(\xi)]:(0,T_+)\times\mathbb{R}\longrightarrow\mathbb{R}$ is a ${\rm C}^\infty$-function. \\[-2ex]
\item[(iib)] For any $k\in\N$, we have $f\in {\rm C}^\infty ((0,T_+), H^k(\mathbb{R})).$\\[-2ex]
\end{itemize} 
\item[(iii)]  {\em (Global existence)} If 
$$\sup_{[0,T]\cap [0,T_+(f_0))} \|f(t)\|_{H^s}<\infty$$
for each $T>0$, then $T_+(f_0)=\infty.$
\end{itemize} 
\end{thm}

\begin{rem}  Observe that the complete problem \eqref{STOKES} is encoded in the time evolution of $f$.
Besides, if $f$ is a solution to \eqref{STOKES}, then, given $\lambda>0$, also $[t\mapsto f_\lambda(t)]$ given by
\[
f_\lambda(t)(\xi):=\lambda^{-1}f(\lambda t)(\lambda\xi),
\]
is a solution to \eqref{STOKES}. 
This identifies~$H^{3/2}(\mathbb{R})$ as a critical space for the evolution problem~\eqref{STOKES}. 
Hence, Theorem~\ref{MT1}   covers all subcritical spaces. 
\end{rem}

\subsection{\label{subsec01}Outline}
The paper is structured as follows: In Section \ref{sec:aux} we discuss a two-phase Stokes problem with equal viscosities in both phases where the normal stresses are continuous across the interface and the velocity has a prescribed jump there. In fact, the problem is solved by the hydrodynamic double layer potential generated by that jump. Although the boundary behavior of this potential is well-known, we prove the results on this in Appendix \ref{neargamma} as they do not seem directly available in the literature for our unbounded geometry. 

As we rely on the solvability of singular integral equations of the second kind arising from the hydrodynamic double-layer  potential,
 the spectrum of the corresponding operator is investigated in Sections \ref{Sec:10} and \ref{Sec:11}, first in $L_2(\R)^2$ and 
 then in~ $H^s(\R)^2,$ with~$s\in(3/2,2)$, and~$H^2(\R)^2$. 
 The main technical  tools in the latter  cases are shift invariances and commutator properties for singular integral operators of the type discussed here. 
 In Section \ref{Sec:13} we reformulate the moving boundary problem~\eqref{STOKES} as a nonlinear and nonlocal evolution equation problem, cf.~\eqref{NNEP}. 
 Finally, in Section \ref{Sec:linfinal} we carry out the linearization of \eqref{NNEP}  and locally approximate the linearization  by Fourier multipliers. 
This enables us to identify the parabolic character of the evolution equation and to prove our main result by invoking abstract results on equations of that type from~\cite{L95}.

\subsection{\label{subsec02}Notation}
 Slightly deviating from the usual notation, if $E_1,\ldots,E_k,\,F$, $k\in\N$, are Banach spaces, we write $\kL^k(E_1,\ldots,E_k;F)$ for the Banach space of bounded $k$-linear maps from $\prod_i E_i$ to $F$.  Given Banach spaces $X$ and $Y$, we let $\kL^k_{\rm sym}(X,Y)\subset\kL^k(X,\ldots,X;Y)$ denote the space of $k$-linear, bounded symmetric maps $A:\;X^k\to Y$.
 Moreover, ${\rm C}^{-1}(E,F)$ will denote the space of locally Lipschitz continuous maps from a Banach space~$E$ to a Banach space~$F$.
Given $k\in\N$, we further let~${\rm C}^k(\R)$ denote the Banach space of functions with bounded and continuous derivatives up to order~$k$ and ${\rm C}^{k+\alpha}(\R)$, $\alpha\in(0,1)$, 
is its subspace consisting of functions  with $\alpha$-H\"older continuous $k$th derivative  whose $\alpha$-H\"older modulus is bounded.

\section{\label{sec:aux} An auxiliary fixed-time problem}
 As a preparation for solving the boundary value problem ~$\eqref{probint}_1-\eqref{probint}_5$ for fixed time, in this section we consider the related Stokes problem \eqref{bvpaux} with equal viscosities normed to~$1$.
  The unique solvability of \eqref{bvpaux} is established  in  Proposition~\ref{auxpropbeta} below and in Appendix~\ref{neargamma}.

In this section, $f\in H^3(\RRM)$ is fixed. We introduce the following notation:
\[\Omega^\pm:=\Omega_f^\pm:=\{(x_1,x_2)\in\mathbb{R}^2\,|\,x_2\gtrless f(x_1)\},
 \qquad \Gamma:=\Gamma_f:=\partial\Omega^\pm= \{(\xi,f(\xi))\,|\,\xi\in\mathbb{R}\}.\]
 Note that $\Gamma$ is the image of $\R$ under the diffeomorphism 
 \[\Xi:=\Xi_f:=({\rm id}_\mathbb{R},f).\]
 Further, let $\nu$ and $\tau$ be the componentwise pull-back under~$\Xi$ of the unit normal~$\tilde\nu$ on~$\Gamma$ exterior to~$\Omega^-$ and of the unit tangent vector $\tilde\tau$ to~$\Gamma$,  that is
 \begin{align}\label{nutau}
\nu:=\tfrac{1}{\omega}(-f',1)^\top,\qquad\tau:=\tfrac{1}{\omega}(1,f')^\top,\qquad \omega:=\omega(f):=(1+f'^2)^{1/2}.
 \end{align}
 For any function $z$ defined on $\RRM^2\setminus\Gamma$ we set $z^\pm:=z|_{\Omega^\pm}$ and if $z^\pm$ have limits at some point~$(\xi,f(\xi))\in\Gamma$ we will write $z^{\pm}(\xi,f(\xi))$ for the limits, 
 and we set
 \be\label{defjump}
 [z] (\xi,f(\xi)):=z^+(\xi,f(\xi))-z^-(\xi,f(\xi)).
\ee
For notational brevity we introduce the function space
\begin{align*}
X:=X_f:=\left\{(w,q):\RRM^2\setminus\Gamma\longrightarrow\RRM^2\times\RRM\,\left|\,
\begin{aligned}
&w^\pm\in {\rm C}^2(\Omega^\pm,\RRM^2)\cap {\rm C}^1(\overline{\Omega^\pm},\RRM^2)\\
&q^\pm\in {\rm C}^1(\Omega^\pm)\cap {\rm C}(\overline{\Omega^\pm})
\end{aligned}\right.\right\}.
\end{align*}
For given $\beta=(\beta_1,\beta_2)^\top\in H^2(\RRM)^2$ we seek solutions $(w,q)\in X$ to the Stokes problem 
\be\label{bvpaux}
\left.\begin{array}{rcll}
\Delta w^\pm-\nabla q^\pm&=&0&\mbox{in $\Omega^\pm$,}\\
\vdiv w^\pm&=&0&\mbox{in $\Omega^\pm$,}\\{}
[w]&=&\beta\circ\Xi^{-1}&\mbox{on $\Gamma$,}\\
{}[T_1(w,q)](\nu\circ \Xi^{-1})&=&0&\mbox{on $\Gamma$,}\\
(w^\pm,q^\pm)(x)&\to&0&\mbox{for $|x|\to\infty$.}
\end{array}\right\}
\ee

For the construction of the solution to~\eqref{bvpaux}, let us first point out that for any smooth solution~${(U,P):E\longrightarrow\RRM^2\times\RRM}$ 
to the homogeneous Stokes system 
\begin{equation}\label{inhstosy}
 \left.\begin{array}{rllll}
\Delta U-\nabla P&=&0,\\[1ex]
\vdiv U&=&0
\end{array}\right\}\qquad\text{in  $E$},
\end{equation}
where $E$ is a domain in $\RRM^2$, the functions $(W^i,Q^i):E\longrightarrow\RRM^2\times \RRM$, $i=1,2$, given by
\[W^i_j:=T_{1,ij}(U,P)=-P\delta_{ij}+\p_iU_j+\p_jU_i,\quad j=1,\,2,\qquad\text{and}\qquad Q^i=2\p_i P\]
are solutions to \eqref{inhstosy} as well. 
In particular, if  $E=\RRM^2\setminus\{0\}$ and
\[(U,P)=(\clu^k,\clp^k):\mathbb{R}^2\setminus\{0\}\longrightarrow\mathbb{R}^2\times\mathbb{R},\qquad  k=1,2,\]
are the fundamental solutions to the Stokes equations \eqref{inhstosy}, given by 
\begin{equation}\label{fundUP}
\clu_j^k(y)=-\frac{1}{4\pi}\left(\delta_{jk}\ln\frac{1}{|y|}+\frac{y_jy_k}{|y|^2}\right),\quad j=1,\,2,\qquad\text{and}\qquad \clp^k(y)=-\frac{1}{2\pi}\frac{y_k}{|y|^2}
\end{equation}
for $y=(y_1,y_2)\in\RRM^2\setminus\{0\},$ we obtain a system $(\clw^{i,k},\clq^{i,k}):\RRM^2\setminus\{0\}
\longrightarrow\RRM^2\times\RRM$, $i,k=1,2$, of solutions to the homogeneous Stokes equations given by
\begin{align*}
    \clw^{i,k}_j(y)&:=(-\clp^k\delta_{ij}+\p_i\clu_j^k+\p_j\clu_i^k)(y)=\frac{1}{\pi}\frac{y_iy_jy_k}{|y|^4},\quad j=1,\,2,\\[1ex]
    \clq^{i,k}(y)&:=2\p_i \clp^k(y)=\frac{1}{\pi}\left(-\frac{\delta_{ik}}{|y|^2}+2\frac{y_iy_k}{|y|^4}\right),\qquad y\in\RRM^2\setminus\{0\}.
\end{align*}
We are going to show that $(w,q):=(w,q)[\beta]$ given by
\begin{align}
    w_j(x)&:=\int_\Gamma\clw^{i,k}_j(x-y)\tilde\nu_i(y)\beta_k(y_1)\,d\Gamma_y:=\int_\RRM\clw^{i,k}_j(r)\nu_i(s)\beta_k(s)\omega(s)\,ds,\quad j=1,\,2,\label{defw1}\\
    q(x)&:=\int_\Gamma\clq^{i,k}(x-y)\tilde\nu_i(y)\beta_k(y_1)\,d\Gamma_y:=\int_\RRM\clq^{i,k}(r)\nu_i(s)\beta_k(s)\,\omega(s)ds\label{defq1}
\end{align}
for  $x\in\RRM^2\setminus\Gamma$ and with $r:=r(x,s):=x-(s,f(s))$
 solves \eqref{bvpaux}. 
 Here and further, we sum over indices appearing twice in a product.
We write this more explicitly as
\begin{equation}\label{defw2}
\begin{aligned}
    w(x)&=\frac{1}{\pi}\int_\RRM\frac{-f'r_1+r_2}{|r|^4}
    \begin{pmatrix}
     r_1^2&r_1r_2\\
     r_1r_2&r_2^2
     \end{pmatrix}\beta\,ds,\\
    q(x)&=\frac{1}{\pi}\int_\RRM\frac{1}{|r|^4}\raisebox{1.1ex}{$(-f'\;\;1)$}
    \begin{pmatrix}
    r_1^2-r_2^2&2r_1r_2\\
    2r_1r_2&r_2^2-r_1^2
    \end{pmatrix}\beta\,ds.
\end{aligned}
\end{equation}
The solution $(w,q)$ is the so-called hydrodynamic double-layer potential generated by the density $\beta\circ\Xi^{-1}$ on $\Gamma$, see \cite{Lad63}.
\begin{prop}\label{auxpropbeta}
The boundary value problem \eqref{bvpaux} has precisely one solution  $(w,q)\in X$. It is given by \eqref{defw1}, \eqref{defq1}.   Moreover, $w^\pm|_{\G}\circ\Xi\in H^2(\R)^2$.
\end{prop}
\begin{proof}
The uniqueness of the solution can be shown as in the proof of \cite[Theorem 2.1]{MP2021}. 
Observe that $w$ and $q$ are defined by integrals of the form
\[(w,q)(x)=\int_\RRM K(x,s)\beta(s)\,ds\]
where for every $\alpha\in\N^2$ we have $\partial^\alpha_x K(x,s)=O(s^{-1})$ for $|s|\to\infty$ and locally uniformly in~${x\in \RRM^2\setminus\Gamma}$.
 This shows that $w$ and $q$ are well-defined by \eqref{defw1} and \eqref{defq1}, and that integration and differentiation with respect to $x$ may be interchanged. As $(\clw^{i,k},\clq^{i,k})$ solve the homogeneous Stokes equations, this also holds for $(w,q)$.

To show the decay of $q$ at infinity we obtain  from the matrix equality
\begin{align*}
    &\frac{1}{\pi|r|^4}\raisebox{1.1ex}{$(-f'\;\;1)$}
  \begin{pmatrix}
    r_1^2-r_2^2&2r_1r_2\\
    2r_1r_2&r_2^2-r_1^2
    \end{pmatrix} =-2\p_s (\clp^2(r)\;\;-\clp^1(r)),
\end{align*}
 via integration by parts
\[q(x)=2\int_\RRM(\clp^2\;\;-\clp^1)(r)\beta'\,ds
=\frac{1}{\pi}\int_\RRM\frac{1}{|r|^2}
(-r_2\;\;r_1)\beta'\,ds.\]
 In view of this representation, \cite[Lemma 2.1]{MBV18} implies $q(x)\to 0$ as $|x|\to\infty$.
 
 In order to prove the decay of $w$ we rewrite
 \begin{align*}
     w(x)&=\frac{1}{2\pi}\int_\RRM\frac{-f'r_1+r_2}{|r|^2}\left(
     I+\frac{1}{|r|^2}\begin{pmatrix}
     r_1^2-r_2^2&2r_1r_2\\
     2r_1r_2&r_2^2-r_1^2
     \end{pmatrix}\right)\beta\,ds\\
     &=\frac{1}{2\pi}\int_\RRM\left(\frac{-f'r_1+r_2}{|r|^2}I
     +\p_s\left[\frac{1}{|r|^2}
     \begin{pmatrix}
     r_1r_2& r_2^2\\
     r_2^2&-r_1r_2
     \end{pmatrix}\right]\right)\beta\,ds\\
     &=\frac{1}{2\pi}\int_\RRM\left(
     \frac{-f'r_1+r_2}{|r|^2}\beta-
     \frac{1}{|r|^2}
     \begin{pmatrix}
     r_1r_2& r_2^2\\
     r_2^2&-r_1r_2
     \end{pmatrix}\beta'\right)\,ds,
\end{align*}
where $I\in\R^{2\times 2}$ is the identity matrix.
 Now  \cite[Lemma 2.1]{MBV18} and \cite[Lemma B.2]{MP2021} imply that indeed~$w(x)\to 0$ for $|x|\to\infty$.

The boundary conditions \eqref{bvpaux}$_3$ and \eqref{bvpaux}$_4$  together with the properties  that $(w,q)\in X$ and~${w^\pm|_{\G}\circ\Xi\in H^2(\R)^2}$ are shown in Appendix \ref{neargamma}.
 \end{proof}

           %%%%%%%%%%%%%%%%%%%%%%%%%%%%%%%%%%%%%%%%%%%%%%%%%%
%%%%%%%%%%%%%%%%%%%%%%%%%%%%%%%%%%%%%%%%%%%%%%%%%%
%%%%%%%%%%%%%%%%%%%%%%%%%%%%%%%%%%%%%%%%%%%%%%%%%%
%%%%%%%%%%%%%%%%%%%%%%%%%%%%%%%%%%%%%%%%%%%%%%%%%%
\section{The $L_2$-resolvent of the hydrodynamic double-layer potential operator}\label {Sec:10}
%%%%%%%%%%%%%%%%%%%%%%%%%%%%%%%%%%%%%%%%%%%%%%%%%%
%%%%%%%%%%%%%%%%%%%%%%%%%%%%%%%%%%%%%%%%%%%%%%%%%%
%%%%%%%%%%%%%%%%%%%%%%%%%%%%%%%%%%%%%%%%%%%%%%%%%%
%%%%%%%%%%%%%%%%%%%%%%%%%%%%%%%%%%%%%%%%%%%%%%%%%%

 In this section we study the  resolvent set  of the hydrodynamic double-layer potential operator $\bD(f),$  with $f\in {\rm C}^1(\R)$, introduced in \eqref{DFB} below,
 which we view in this section as an element of~$\kL(L_2(\R)^2)$.
The main result of this section  is  Theorem~\ref{L2spec} below which provides in particular the invertibility of $\lambda-\bD(f)$ for $\lambda\in\R$ with $|\lambda|>1/2$.
 
To begin,  we introduce a general class of singular integral  operators  suited to our approach via layer potentials, cf. \cite{MBV19,MP2021}. 
Given~${n,\,m\in\N}$ and  Lipschitz continuous  functions~${a_1,\ldots, a_{m},\, b_1, \ldots, b_n:\mathbb{R}\longrightarrow\mathbb{R}}$,  we let~$B_{n,m}$ denote  the singular integral   operator
\begin{equation}\label{BNM}
 B_{n,m}(a_1,\ldots, a_m)[b_1,\ldots,b_n,h](\xi):=\PV\int_\mathbb{R}  \frac{h(\xi-\eta)}{\eta}\cfrac{\prod_{i=1}^{n}\big(\delta_{[\xi,\eta]} b_i /\eta\big)}{\prod_{i=1}^{m}\big[1+\big(\delta_{[\xi,\eta]}  a_i /\eta\big)^2\big]}\, d\eta,
\end{equation}
where  $\PV\int_\R$ denotes the principal value integral and $\delta_{[\xi,\eta]}u:=u(\xi)-u(\xi-\eta)$. 
 For brevity we set
\begin{equation}\label{defB0}
B^0_{n,m}(f):=B_{n,m}(f,\ldots f)[f,\ldots,f,\cdot].
\end{equation}

In   this section we several times use the following result.
\begin{lemma}\label{L:MP0'}
There exists a constant~$C$ depending only 
on $n,\, m$, and $\max_{i=1,\ldots, m}\|a_i'\|_{\infty}$ with
 $$\|B_{n,m}(a_1,\ldots, a_m)[b_1,\ldots,b_n,\,\cdot\,]\|_{\kL(L_2(\mathbb{R}))}\leq C\prod_{i=1}^{n} \|b_i'\|_{\infty}.$$ 
 Moreover,   $B_{n,m}\in {\rm C}^{1-}(W^1_\infty(\mathbb{R})^{m},\kL_{{\rm sym}}^n(W^1_\infty(\mathbb{R}) ,\kL(L_2(\mathbb{R})))).$
\end{lemma}
\begin{proof}
See  \cite[Remark 3.3]{MBV19}.
\end{proof}

Given $f\in  {\rm C^1}(\R)$, we introduce the linear operators $\bD(f)$ and $\bD(f)^\ast$ defined  by 
\begin{equation}\label{defD}
\begin{aligned}
    \bD(f)[\beta](\xi)&:=\frac{1}{\pi}\PV\int_\RRM
    \frac{r_1 f'- r_2}{ |r|^4}
   \begin{pmatrix}
r_1^2&r_1 r_2\\
r_1 r_2& r_2^2
\end{pmatrix}\beta\,ds,\\
\bD(f)^\ast[\beta](\xi)&:=\frac{1}{\pi}\PV\int_\RRM
    \frac{-r_1 f'(\xi)+ r_2}{|r|^4}
\begin{pmatrix}
r_1^2&r_1 r_2\\
 r_1r_2&r_2^2
\end{pmatrix}\beta\,ds,
\end{aligned}
\end{equation}
where $\xi\in\RRM$ and  $\beta\in L_2(\RRM)^2$. Throughout this section $r:=(r^1,r^2)$ is given by
\begin{align}\label{rxs}
r:=r(\xi,s):=(\xi-s,f(\xi)-f(s)).
\end{align}
 We note that $\bD(f)$ is related to the  $B_{n,m}$  via
\begin{align}\label{DFB}
\bD(f)[\beta] 
=\frac{1}{\pi}\begin{pmatrix}
B_{0,2}^0(f)&B_{1,2}^0(f)\\[1ex]
B_{1,2}^0(f)&B_{2,2}^0(f)
\end{pmatrix}
\begin{pmatrix}
f'\beta_1\\[1ex]
f'\beta_2
\end{pmatrix}
-\frac{1}{\pi}\begin{pmatrix}
B_{1,2}^0(f)&B_{2,2}^0(f)\\[1ex]
B_{2,2}^0(f)&B_{3,2}^0(f)
\end{pmatrix}
\begin{pmatrix}
\beta_1\\[1ex]
\beta_2
\end{pmatrix}
\end{align} 
for $\beta=(\beta_1,\,\beta_2)^\top$. 
Therefore, as a direct consequence of Lemma~\ref{L:MP0'},~$\bD(f)$ is bounded on~$L_2(\RRM)^2$.
Moreover, up to the sign and the push-forward via~$\Xi$, $\bD(f)[\beta](\xi)$ is the ``direct value'' of 
the hydrodynamic double-layer potential~$w$ generated by $\beta$  in $(\xi,f(\xi))\in\Gamma$, cf.~$\eqref{defw2}_1$. 
One may also check that $\bD(f)^*$ is the $L_2$-adjoint~of~$\bD(f)$.

Using the same notation, we define the  singular integral operators~$\bB_1(f)$ and~$\bB_2(f)$ by
\begin{align*}
    \bB_1(f)[\theta](\xi):=\frac{1}{\pi}\PV\int_\RRM
    \frac{- r_1 f'+ r_2}{ |r|^2}\,\theta\,ds\quad\text{and}\quad
    \bB_2(f)[\theta](\xi):=\frac{1}{\pi}\PV\int_\RRM
    \frac{r_1 + r_2 f'}{|r|^2}\,\theta\,ds,
\end{align*}
where  $\theta\in L_2(\RRM)$.
The operators $\bB_i(f)$, $i=1,\, 2$, play an important role also in the study of the Muskat problem, cf.~\cite{MBV18}.
Lemma~\ref{L:MP0'} implies that  also  $\bB_i(f)$ $i=1,\, 2$, is    bounded on~$L_2(\RRM)$.
 Moreover, 
  $\bB_1(f)[\theta](\xi)$ is the direct value of the double layer potential for the Laplacian corresponding to the density $\theta$ in $(\xi,f(\xi))\in\Gamma$.

We are going to prove in Theorem~\ref{L2spec} below that the resolvent sets of $\bD(f)$ and $\bD(f)^\ast$ contain all real~$\lambda$ with~$|\lambda|>1/2$, 
with a bound on the resolvent that is uniform in $\lambda$ away from $\pm1/2$, and in $f$ as long as $\|f'\|_\infty$ is bounded. 

 Oriented at \cite{FKV88,CP09}, we obtain this property  on the basis of a Rellich identity for the Stokes operator.
While eventually the result for $\bD(f)$ is needed, it is helpful to consider $\bD(f)^\ast$, as this operator naturally arises from the jump relations for the single-layer hydrodynamic potential generated by $\beta$,  cf.~\eqref{stressbdry} below.

 We next derive the Rellich identity~\eqref{rellich}, and based on it we establish  an estimate that relates the operator $\bD(f)^\ast$ to the operators 
$\bB_1(f)$ and $\bB_2(f)$ introduced above.
\begin{lemma}\label{lemrel}
Given  $K>0$,  there exists a positive constant $C$,  that depends only on $K$, such that for all~${\beta\in L_2(\RRM)^2}$, $\lambda\in[-K,K]$,  
and~$f\in  {\rm C}^1(\RRM)$ with  $\|f'\|_\infty<K$ we have
\be\label{lemma-rel}
\begin{aligned}
C\|(\lambda-\bD(f)^*)[\beta]\|_2\|\beta\|_2&\geq\|(\lambda-\tfrac{1}{2}\bB_1(f))
[\omega^{-1}\beta\cdot\nu]-\tfrac{1}{2}\bB_2(f)[\omega^{-1}\beta\cdot\tau]\|^2_2\\[1ex]
&\hspace{0.45cm}+m(\lambda)\|\omega^{-1}\beta\cdot\tau\|^2_2,
\end{aligned}
\ee
where $\omega$, $\nu$, and $\tau$ are defined in \eqref{nutau}, and with
\be\label{defmlam}
m(\lambda):=\max\left\{
\left(\lambda+\tfrac{1}{2}\right)\left(\lambda-\tfrac{3}{2}\right),
\left(\lambda-\tfrac{1}{2}\right)\left(\lambda+\tfrac{3}{2}\right)
\right\}.
\ee
\end{lemma}
\begin{proof}
Let first  $f\in {\rm C}^\infty(\R)$ and $\beta=(\beta_1,\beta_2)^\top$ with $\beta_k\in {\rm C}_0^\infty(\R)$, $k=1,\, 2$. 
 We define the hydrodynamic single-layer potential $u$ with corresponding pressure $\Pi$ by 
\[ u(x):=-\int_\R\cU^k( x-(s,f(s)))\beta_k(s)\, ds\qquad\text{and}\qquad \Pi(x):=-\int_\R\cP^k( x-(s,f(s)))\beta_k(s)\, ds\]
for $x\in\R^2\setminus\Gamma$, where  and   $\cU^k$, $\cP^k$ defined by \eqref{fundUP}. 
 Using the fact that  $\beta$ is  compactly supported, is is not difficult to see that the functions $(u,\Pi)$ are well-defined and smooth in~$\Omega^\pm$ and satisfy 
\begin{equation}\label{Sto34}
 \left.\begin{array}{rllll}
\Delta u-\nabla \Pi&=&0,\\[1ex]
\vdiv u&=&0
\end{array}\right\}\qquad\text{in  $\0^\pm$},
\end{equation}
as well as
\begin{equation}\label{decay}
\Pi,\,\nabla u=O(|x|^{-1})\qquad \mbox{for $|x|\to\infty$}.
\end{equation}
Moreover,  \cite[Lemma~A.1]{BM21x} and  the arguments in the proof of \cite[Lemma~A.1]{MP2021} show that
the functions $\Pi|_{\Omega^\pm}$ and $u|_{\Omega^\pm}$ have extensions $\Pi^\pm\in {\rm C}(\overline{\Omega^\pm})$ and 
$u^\pm\in {\rm C}^1(\overline{\Omega^\pm})$, and, given $\xi\in\R$, we have 
\begin{equation}\label{repp}
\begin{aligned}
    \partial_iu_j^\pm \circ\Xi(\xi) &=-\PV\int_\RRM\partial_i\cU_j^k(r)\beta_k\,ds\pm\frac{-\beta_j\nu^i+\nu^i\nu^j\beta\cdot\nu}{2\omega}(\xi),\\
    \Pi^\pm\circ\Xi(\xi) &= -\PV\int_\RRM\cP^k(r)\beta_k\,ds\pm\frac{\beta\cdot\nu}{2\omega}(\xi)\\
    &=\frac{1}{2}\bB_1(f)[\omega^{-1}\beta\cdot\nu](\xi)+
    \frac{1}{2}\bB_2(f)[\omega^{-1}\beta\cdot\tau](\xi)
    \pm\frac{\beta\cdot\nu}{2\omega}(\xi),
\end{aligned}
\end{equation}
  where $\nu=(\nu^1,\nu^2)$ and $r=r(\xi,s)$  are defined  in \eqref{nutau} and \eqref{rxs}.
In particular, 
 \be\label{d2urep}
 \partial_2u^\pm\circ\Xi(\xi)=\TTM(f)[\beta](\xi)\mp\frac{(\beta\cdot\tau)\tau}{2\omega^2}(\xi),
 \ee
 where  $\TTM(f)$ is the singular integral operator given by
 \[\TTM(f)[\beta](\xi):=\frac{1}{4\pi}\PV\int_\RRM\frac{1}{|r|^4}
\begin{pmatrix}
 - r_2^3-3r_1^2r_2& r_1^3- r_1 r_2^2\\
r_1^3-r_1 r_2^2&  r_1^2 r_2- r_2^3
 \end{pmatrix}
 \left(\begin{array}{c}\beta_1\\\beta_2\end{array}\right)\,ds.\]
 Observe that $\TTM(f)$ is skew-adjoint on $L_2(\RRM)^2$, i.e. $\TTM(f)^\ast=-\TTM(f)$, and therefore
 \be\label{Tskew}
 \langle\TTM(f)[\beta]\,|\,\beta\rangle_2=0.
 \ee
Here $\langle\cdot\,|\,\cdot\rangle_2$ denotes the inner product of $L_2(\RRM)^2$.

Moreover, for the normal stress at the boundary we find
\be\label{stressbdry}
\omega( T_1(u,\Pi)^\pm\circ\Xi)\nu=\Big(\mp\frac{1}{2}-\bD(f)^\ast\Big)[\beta].
\ee
For convenience we introduce the notation 
\[\tau_{ij}:=(T_1(u,\Pi))_{ij}=-\Pi\delta_{ij}+\partial_i u_j+\partial_j u_i,\qquad i,j=1,2,\]
and observe that due to \eqref{Sto34}  
\[\partial_i\tau_{ij}=0 \quad \mbox{in $\Omega^\pm$, $j=1,\, 2,$}\qquad\text{and}\qquad\delta_{ij}\partial_iu_j=0\qquad \text{in $\Omega^\pm$.}\]
The latter identities lead us to 
\[\partial_i(\tau_{ij}\partial_2u_j)=\tau_{ij}\partial_i\partial_2u_j
=(\partial_iu_j+\partial_ju_i)\partial_2\partial_iu_j=\frac{1}{4} \sum_{i,\,j=1}^2\partial_2(\partial_iu_j+\partial_ju_i)^2\qquad \text{in $\Omega^\pm$.}\]
 In view of \eqref{decay} we may  integrate the latter relation  over $\Omega^\pm$ and using  Gauss' theorem  and~\eqref{stressbdry} we get
\be\label{rellich}
\int_\Gamma\frac{1}{\omega} \sum_{i,\,j=1}^2(\partial_iu_j^\pm+\partial_ju_i^\pm)^2\,d\Gamma
=4\int_\Gamma\tau_{ij}^\pm\tilde\nu_i\partial_2u_j^\pm\,d\Gamma
=4\Big\langle\Big(\mp\frac{1}{2}-\bD(f)^\ast\Big)[\beta]\,\Big|\,\partial_2 u^\pm\circ\Xi\Big\rangle_2.
\ee

To estimate the term on the left we observe that the Cauchy-Schwarz inequality and~${|\tilde\nu|=1}$ yield
\[ \sum_{i,\,j=1}^2(\partial_iu_j^\pm+\partial_ju_i^\pm)^2\geq
 \sum_{i=1}^2((\partial_iu_j^\pm+\partial_ju_i^\pm)\tilde\nu_j)^2
=\sum_{i=1}^2(\tau_{ij}^\pm\tilde\nu_j+\Pi^\pm\tilde\nu_i)^2\qquad
\text{on $\Gamma$.}\]
This inequality, the  representations   \eqref{repp} and  \eqref{stressbdry}, and    $\|\bB_i(f)\|_{\kL(L_2(\R))}\leq  C(K)$, $i=1,\, 2$,  cf. Lemma~\ref{L:MP0'},  now yield
\begin{align*}
&\hspace{-0.5cm}\int_\Gamma\frac{1}{\omega}\sum_{i,j=1}^2(\partial_iu_j^\pm+\partial_ju_i^\pm)^2\,d\Gamma\\[1ex]
&\geq\Big\|\frac{1}{\omega}\Big(\mp\frac{1}{2}-\bD(f)^\ast\Big)[\beta]+(\Pi^\pm\circ\Xi)\nu \Big\|_2^2\\[1ex]
&=\Big\|\frac{1}{\omega}\Big(\lambda-\bD(f)^\ast\Big)[\beta]
-\frac{1}{\omega}\Big(\lambda\pm\frac{1}{2}\Big)\beta
+(\Pi^\pm\circ\Xi)\nu \Big\|_2^2\\[1ex]
&\geq\Big\|-\frac{1}{\omega}\Big(\lambda\pm\frac{1}{2}\Big)\beta
+\Big(\frac{1}{2}\bB_1(f)[\omega^{-1}\beta\cdot\nu]
+\frac{1}{2}\bB_2(f)[\omega^{-1}\beta\cdot\tau]\pm\frac{\beta\cdot\nu}{2\omega}\Big)\nu\Big\|_2^2\\[1ex]
&\hspace{0.45cm}+\Big\|\frac{1}{\omega}(\lambda-\bD(f)^\ast)[\beta]\Big\|_2^2
-C\|(\lambda-\bD(f)^\ast)[\beta]\|_2\|\|\beta\|_2
\\[1ex]
&\geq\Big(\lambda\pm\frac{1}{2}\Big)^2\|\omega^{-1}\beta\cdot\tau\|_2^2
+\Big\|\Big(\lambda-\frac{1}{2}\bB_1(f)\Big)[\omega^{-1}\beta\cdot\nu]
-\frac{1}{2}\bB_2(f)[\omega^{-1}\beta\cdot\tau]\Big\|_2^2\\[1ex]
&\hspace{0.45cm}-C\|(\lambda-\bD(f)^\ast)[\beta]\|_2\|\|\beta\|_2
\end{align*}
for any $\lambda\in[-K,K]$. 

We next consider the term on the right of \eqref{rellich}.   
As a direct consequence of Lemma~\ref{L:MP0'}  we note that  $\|\TTM(f)\|_{\kL(L_2(\R)^2)}\leq C=C(K)$.
This bound together with~\eqref{d2urep} and~\eqref{Tskew}  implies
\begin{align*}
    4\Big\langle\Big(\mp\frac{1}{2}-\bD(f)^\ast\Big)[\beta]\,\Big|
    \,\partial_2u\circ\Xi\Big\rangle_ 2
    &= 4\Big\langle\Big.\Big(\lambda-\bD(f)^\ast\Big)[\beta]
    -\Big(\lambda\pm\frac{1}{2}\Big)\beta\,\Big|\TTM[\beta]\mp
    \frac{(\beta\cdot\tau)\tau}{2\omega^2}\Big\rangle_2\\[1ex]
    &\leq C\|(\lambda-\bD(f)^\ast)[\beta]\|_2\|\|\beta\|_2
    \pm2\Big(\lambda\pm\frac{1}{2}\Big)\|\omega^{-1}\beta\cdot\tau\|_2^2.
\end{align*}
 For $f\in {\rm C}^\infty(\R)$, the estimate~\eqref{lemma-rel} follows  from \eqref{rellich} and the latter estimates upon rearranging terms and a standard density argument.
 For general $f\in{\rm C^1}(\R)$ we additionally need to use the continuity of the mappings 
 \[[f\mapsto \bD(f)^*]:{\rm C^1}(\R)\to\kL(L_2(\R)^2)\qquad\text{and}\qquad [f\mapsto \bB_{i}(f)]:{\rm C^1}(\R)\to\kL(L_2(\R)), \,\, i=1,\, 2,\]
 which are direct consequences of Lemma~\ref{L:MP0'}, together  with the density of  $ {\rm C}^\infty(\R)$ in $ {\rm C^1}(\R)$.
\end{proof}

Based on Lemma~\ref{lemrel} we now establish the following result.   
\begin{thm}[Spectral properties of $\bD(f)$ and $\bD(f)^\ast$]\label{L2spec}
Given  $\delta\in(0,1)$, there exists a constant~$C=C(\delta)>0$ such that for all $\lambda\in\RRM$ with $|\lambda|\geq1/2+\delta$ and~$f\in {\rm C}^1(\RRM)$ with~${\|f'\|_\infty\leq 1/\delta}$ we have 
\begin{align}\label{DEest}
\|(\lambda-\bD(f)^\ast)[\beta]\|_2\geq  C\|\beta\|_2\qquad\text{for all $\beta\in L_2(\RRM)^2$}.
\end{align}
Moreover, $\lambda-\bD(f)^\ast$, $\lambda-\bD(f)\in \kL(L_2(\RRM)^2)$ are isomorphisms for all  $\lambda\in\RRM$ with $|\lambda|>1/2$ and~$f\in {\rm C}^1(\RRM)$.
\end{thm}
\begin{proof}
In order to prove \eqref{DEest} we assume  the opposite. 
Then we may find sequences $(\lambda_k)$ in~$\RRM$, $(f_k)$ in~${\rm C}^1(\RRM)$, and $(\beta_k)$ in  $L_2(\RRM)^2$ such that $|\lambda_k|\geq 1/2+\delta,$ $\|f_k'\|_\infty\leq 1/\delta$,
 and~$\|\beta_k\|_2=1$ for all $k\in\N$, and
\[(\lambda_k-\bD(f_k)^\ast)[\beta_k]\to 0\qquad \mbox{ in $L_2(\RRM)^2$.}\]
Given $k\in\N,$ we set $\omega_k:=\omega(f_k)$,  \eqref{nutau}.
 As the operators $\bD(f_k)^\ast$ are bounded, uniformly in~${k\in\N}$, in $\kL(L_2(\R)^2)$, cf. Lemma~\ref{L:MP0'},
 the sequence $(\lambda_k)$
is bounded.
 Observing that for the constant $m=m(\lambda)$ from~\eqref{defmlam} we have $m(\lambda_k)\geq  \delta(2+\delta)>0$ for all $k\in\N$, 
  we get from Lemma~\ref{lemrel}   that
\[\omega_k^{-1}\beta_k\cdot\tau\to 0,\quad \big(\lambda_k-\tfrac{1}{2}\bB_1(f_k)\big)[\omega_k^{-1}\beta_k\cdot\nu]
-\tfrac{1}{2}\bB_2(f_k)[\omega_k^{-1}\beta_k\cdot\tau]\to 0 \qquad\mbox{ in $L_2(\RRM)$.}\]
As the operators $\bB_2(f_k)$ are  bounded, uniformly in $k\in\N$, in $\kL(L_2(\R)^2)$, cf. Lemma~\ref{L:MP0'}, this implies
\[\big(\lambda_k-\tfrac{1}{2}\bB_1(f_k)\big)[\omega_k^{-1}\beta_k\cdot\nu]\to 0\qquad\text{in $L_2(\RRM)$.}\]
Let $\bA(f):=\bB_1(f)^\ast$. 
Since $|2\lambda_k|\geq 1$, it follows from the proof of~\cite[Theorem~3.5]{MBV18} that the operator
$2\lambda_k-\bA(f_k)\in\kL(L_2(\R))$, $k\in\N$, is an isomorphism with
\[\|\left(2\lambda_k-\bA(f_k)\right)^{-1}\|_{\kL(L_2(\R))}\leq C(\delta).\]
This implies that also $2\lambda_k- \bB_1(f_k)\in\kL(L_2(\RRM))$, $k\in\N$, is an isomorphism  and
\[ \big\|(\lambda_k- \tfrac{1}{2}\bB_1(f_k))^{-1}\big\|_{ \kL(L_2(\R))}\leq C(\delta).\]
Thus $\omega_k^{-1}\beta_k\cdot\nu\to 0$ in $L_2(\RRM)$, so that
\[\beta_k=\omega_k\big(\omega_k^{-1}(\beta_k\cdot\nu)\nu+\omega_k^{-1}(\beta_k\cdot\tau)\tau\big)\to 0 \quad\mbox{ in $L_2(\RRM)^2$.}\]
This contradicts the property that $\|\beta_k\|_2=1$ for all $k\in\N$ and \eqref{DEest} follows.

To complete the proof we fix $f\in {\rm C}^1(\RRM)$ and $\lambda_0\in\RRM$ with $|\lambda_0|>1/2$ and we choose~${\delta\in(0,1)}$ such that  $|\lambda_0|\geq 1/2+\delta$ and   $\|f'\|_\infty\leq 1/\delta$.
As $\bD(f)^\ast$ is bounded, $\lambda-\bD(f)^\ast\in\kL(L_2(\RRM)^2)$ is an isomorphism if $|\lambda|$ is sufficiently large. 
The estimate \eqref{DEest} together with a standard continuity argument, cf. e.g. \cite[Proposition~I.1.1.1]{Am95},   now implies that $\lambda_0-\bD(f)^*$ is an  isomorphism as well.
 The result  for~$\bD(f)$ is an immediate consequence of this property. 
\end{proof}
 
%%%%%%%%%%%%%%%%%%%%%%%%%%%%%%%%%%%%%%%%%%%%%%%%%%
%%%%%%%%%%%%%%%%%%%%%%%%%%%%%%%%%%%%%%%%%%%%%%%%%%
%%%%%%%%%%%%%%%%%%%%%%%%%%%%%%%%%%%%%%%%%%%%%%%%%%
%%%%%%%%%%%%%%%%%%%%%%%%%%%%%%%%%%%%%%%%%%%%%%%%%%
\section{The  resolvent of the hydrodynamic double-layer potential operator  in higher order Sobolev spaces}\label {Sec:11}
%%%%%%%%%%%%%%%%%%%%%%%%%%%%%%%%%%%%%%%%%%%%%%%%%%
%%%%%%%%%%%%%%%%%%%%%%%%%%%%%%%%%%%%%%%%%%%%%%%%%%
%%%%%%%%%%%%%%%%%%%%%%%%%%%%%%%%%%%%%%%%%%%%%%%%%%
%%%%%%%%%%%%%%%%%%%%%%%%%%%%%%%%%%%%%%%%%%%%%%%%%%
The main goal of this section is to establish spectral properties for $\bD(f)$, parallel to those in Theorem \ref{L2spec}, in the spaces $H^{s-1}(\RRM)^2$, $s\in(3/2/2)$, 
and in $H^2(\RRM)^2$. 
The latter are needed  when solving the  fixed-time problem  \eqref{probintFT}, see Proposition~\ref{P:STOsol}, and the former are used to  derive and  study  
the  contour integral formulation~\eqref{NNEP} of the evolution problem~\eqref{STOKES}.

For this purpose, we first recall some further results on the  singular integral  operators  $B_{n,m}$ introduced in \eqref{BNM}. 
\begin{lemma}\label{L:MP0}
$$
$$

\vspace{-0.7cm}
\begin{itemize}
 \item[(i)] Let   $n\geq1,$  $s\in(3/2,2),$  and  $a_1,\ldots, a_m\in H^s(\RRM)$ be given.  
  Then, there exists a constant~$C$, depending only on $n,\, m$, $s$,  and $\max_{1\leq i\leq m}\|a_i\|_{H^s}$, such that
\begin{align} 
&\| B_{n,m}(a_1,\ldots, a_{m})[b_1,\ldots, b_n,h]\|_2\leq C\|b_1\|_{H^1}\|h\|_{H^{s-1}}\prod_{i=2}^{n}\|b_i\|_{H^s} \label{REF1}
\end{align}
for all $b_1,\ldots, b_n\in H^s(\R)$ and $h\in H^{s-1}(\R).$

Moreover,   $B_{n,m}\!\in\! {\rm C}^{1-}(H^s(\R)^m,\kL^{n+1}( H^1(\R), H^{s}(\R),\ldots,H^s(\R),H^{s-1}(\R); L_2(\R))).$\\[-1ex] 

 \item[(ii)] Given $s\in(3/2 ,2)$ and $a_1,\ldots, a_m \in H^s(\mathbb{R})$, there exists a constant C,
  depending  only on $n,\, m,\, s$,  and $\max_{1\leq i\leq m}\|a_i\|_{H^s},$ such that
\begin{align*} 
\| B_{n,m}(a_1,\ldots, a_{m})[b_1,\ldots, b_n,h]\|_{H^{s-1}}\leq C \|h\|_{H^{s-1}}\prod_{i=1}^{n}\|b_i\|_{H^{s}}
\end{align*}
for all $ b_1,\ldots, b_n\in H^s(\mathbb{R})$ and $h\in H^{s-1}(\mathbb{R}). $

Moreover,  $  B_{n,m}\in {\rm C}^{1-}(H^s(\mathbb{R})^m,\kL^{n}_{\rm sym}( H^s(\mathbb{R}) , \kL(H^{s-1}(\mathbb{R})))).$ \\[-1ex]

\item[(iii)] Let $n\geq 1$, $3/2<s'<s<2$, and $a_1,\ldots, a_m \in H^s(\mathbb{R})$ be given. 
  Then, there exists a constant  $C$, depending only on $n,\, m$, $s$, $s'$,  and $\max_{1\leq i\leq m}\|a_i\|_{H^s}$, such that
\begin{equation*} 
\begin{aligned} 
&\| B_{n,m}(a_1,\ldots, a_{m})[b_1,\ldots, b_n,h] -h B_{n-1,m}(a_1,\ldots, a_{m})[b_2,\ldots, b_n,b_1']\|_{H^{s-1}}\\[1ex]
&\hspace{3cm}\leq C \|b_1\|_{H^{s'}}\|h\|_{H^{s-1}}\prod_{i=2}^{n}\|b_i\|_{H^s}
\end{aligned}
\end{equation*}
for all $b_1,\ldots, b_n\in H^s(\mathbb{R})$ and $h\in H^{s-1}(\mathbb{R}).$
\end{itemize}
\end{lemma}
\begin{proof}
The claims  (i) is established in \cite[Lemmas~3.2]{MBV18}, while (ii) and (iii) are proven in \cite[Lemma~5 and~ Lemma 6]{AM21x}.
\end{proof}

For $\xi\in\RRM$ we define the left shift operator $\tau_\xi$ on $L_2(\RRM)$ by $\tau_\xi u(x):=u(x+\xi)$ and observe the invariance property
\begin{equation}
    \label{invar}
\tau_\xi B_{n,m}(a_1,\ldots,a_m)[b_1,\ldots,b_n,h]
=B_{n,m}(\tau_\xi a_1,\ldots,\tau_\xi a_m)[\tau_\xi b_1,\ldots,\tau_\xi b_n,
\tau_\xi h].
\end{equation}
Differences of $B_{n,m}$ with respect to the nonlinear arguments $a_i$ can be represented by the identity
\begin{equation}\label{diffid}
\begin{aligned}
&B_{n,m}(a_1,a_2\ldots,a_m)[b_1,\ldots,b_n,\cdot]-
B_{n,m}(\tilde a_1,a_2\ldots,a_m)[b_1,\ldots,b_n,\cdot]\\\
&\hspace{0.5cm}=B_{n+2,m+1}( \tilde a_1,a_1,a_2\ldots,a_m)[b_1,\ldots,b_n,\tilde a_1+a_1, \tilde a_1-a_1,\cdot].
\end{aligned}
\end{equation}
We will also use the interpolation property
\begin{align}\label{IP}
[H^{s_0}(\mathbb{R}),H^{s_1}(\mathbb{R})]_\theta=H^{(1-\theta)s_0+\theta s_1}(\mathbb{R}),\qquad\theta\in(0,1),\, -\infty< s_0\leq s_1<\infty,
\end{align}
where $[\cdot,\cdot]_\theta$ denotes the complex interpolation functor of exponent $\theta$.

\begin{thm}\label{T:isom2}
Given $\delta\in(0,1)$ and $s\in(3/2,2)$, there exists a constant $C=C(\delta,s)>0$ such that  
\begin{align}\label{Cdeltas}
\|(\lambda-\bD(f))[\beta]\|_{H^{s-1}}\geq C\|\beta\|_{H^{s-1}}
\end{align}
for all   $\lambda\in\R$ with $|\lambda|\geq 1/2+\delta$, $f\in H^s(\R)$ with $\|f\|_{H^s}\leq 1/\delta,$ and $\beta\in H^{s-1}(\R)^2$.

Moreover,  $\lambda-\bD(f)\in\kL(H^{s-1}(\R)^2)$ is an isomorphism  for all  $\lambda\in\R$ with~$|\lambda|> 1/2$ and~$f\in H^s(\R)$.
\end{thm}

\begin{proof} 
Given $f\in H^s(\RRM)$, the relation \eqref{DFB} and Lemma \ref{L:MP0}~(ii) imply  ${\bD(f)\in\kL(H^{s-1}(\R)^2)}$.
In order to prove \eqref{Cdeltas}, let $\lambda\in\R$ with $|\lambda|\geq 1/2+\delta$ and $f\in H^s(\R)$ with $\|f\|_{H^s}\leq 1/\delta $ be fixed. 
Theorem~\ref{L2spec} together with the embedding $H^s(\R)\hookrightarrow L_\infty(\R)$ implies there exists~${C=C(\delta)>0}$ such   
that $\|(\lambda-\bD(\tau_\xi f))^{-1}\|_{\kL(L_2(\R)^2)}\leq C$ for all $\xi\in\R$.
It is well-known there exists a constant~${C>0}$ such that   
$$[\beta]_{H^{s-1}}:=\|[\xi\mapsto |\xi|^{s-1}\mathcal{F}[\beta](\xi)]\|_2=C\Big(\int_{\R}\frac{\|\beta-\tau_{\xi}\beta\|_2^2}{|\xi|^{1+2(s-1)}}\, d{\xi}\Big)^{1/2}=:[\beta]_{W^{s-1}_2},$$ 
where  $\mathcal{F}[\beta] $ is the Fourier transform of $\beta$.
Together with  \eqref{invar} we then get 
\begin{equation}\label{DSD0}
\begin{aligned}
{[\beta]}^2_{ H^{s-1}}&\leq C\int_{\R}\frac{\|(\lambda-\bD(\tau_\xi f))[\beta-\tau_{\xi}\beta]\|_2^2}{|\xi|^{1+2(s-1)}}\, d{\xi}\\[1ex]
&\leq C\Big(\int_{\R}\frac{\|(\lambda-\bD( f))[\beta]-\tau_{\xi}((\lambda-\bD( f))[\beta])\|_2^2}{|\xi|^{1+2(s-1)}}\, d{\xi}
+\int_{\R}\frac{\|(\bD(f) -\bD(\tau_{\xi}f))[\beta]\|_2^2}{|\xi|^{1+2(s-1)}}\, d{\xi}\Big)\\[1ex]
&= C[(\lambda-\bD( f))[\beta]]^2_{H^{s-1}}+C\int_{\R}\frac{\|(\bD(f) -\bD(\tau_{\xi}f))[\beta]\|_2^2}{|\xi|^{1+2(s-1)}}\, d{\xi}.
\end{aligned}
\end{equation}

The term $\|(\bD(f) -\bD(\tau_{\xi}f))[\beta]\|_2$ can be estimated by a finite sum of terms of the form
\[
\|(B_{n,2}^0(f)-B_{n,2}^0(\tau_\xi f))[\beta_i]\|_2\qquad \text{and}\qquad \|B_{n,2}^0(f)[f'\beta_i]-B_{n,2}^0(\tau_\xi f)[(\tau_\xi f')\beta_i]\|_2,
\]
where $0\leq n\leq 3$ and $i\in\{1,2\}$. Let $s'\in(3/2,s)$ be fixed. 
We first consider terms of the second type and estimate  in view of Lemma \ref{L:MP0'} 
\begin{equation}\label{DSD1}
\begin{aligned}
&\hspace{-0.5cm}\|B_{n, 2}^0(f)[f'\beta_i]-B_{n,2}^0(\tau_\xi f)[(\tau_\xi f')\beta_i]\|_2\\[1ex]
&\leq \|(B_{n,2}^0(f)-B_{n,2}^0(\tau_\xi f))[f'\beta_i]\|_2+\|B_{n,2}^0(\tau_\xi f)[(\tau_\xi f'-f')\beta_i]\|_2\\[1ex]
&\leq \|(B_{n,2}^0(f)-B_{n,2}^0(\tau_\xi f))[f'\beta_i]\|_2+C\|\tau_\xi f'-f'\|_2\|\beta\|_{H^{s'-1}}.
\end{aligned}
\end{equation}
Furthermore,  using \eqref{diffid}, we have
\begin{align*}
    B_{n,2}^0(f)-B_{n,2}^0(\tau_\xi f)
    =& \sum_{\ell=1}^n  B_{n,2}(f,f)[\underbrace{\tau_\xi f,\ldots,\tau_\xi f}_{ \ell-1 {\ \rm times}},f-\tau_\xi f,f,\ldots, f,\cdot]\\
    &+ B_{n+2,3}(\tau_\xi f,f,f)[\tau_\xi f,\ldots, \tau_\xi f, \tau_\xi-f,\tau_\xi f+f,\cdot]\\
    &+ B_{n+2,3}(\tau_\xi f,\tau_\xi f,f)[\tau_\xi f,\ldots, \tau_\xi f,\tau_\xi f-f,\tau_\xi f+f,\cdot],
\end{align*}
and together with Lemma \ref{L:MP0}~(i) (with $s'$ instead of $s$), we conclude that $B_{n,2}^0(f)-B_{n,2}^0(\tau_\xi f)$ belongs to $\kL(H^{s'-1}(\R),L_2(\RRM))$ and satisfies
\[
\|B_{n,2}^0(f)-B_{n,2}^0(\tau_\xi f)\|_{\kL(H^{s'-1}(\R),L_2(\RRM))}\leq C\|f-\tau_\xi f\|_{H^1(\RRM)}.
\]
Combining this estimate  with \eqref{DSD1} we get
\[\int_{\R}\frac{\|(\bD(f) -\bD(\tau_{\xi}f))[\beta]\|_2^2}{|\xi|^{1+2(s-1)}}\, d{\xi}\leq C\|f\|_{H^s}^2\|\beta\|_{H^{s'-1}}^2,\]
and by \eqref{DSD0} and the interpolation property \eqref{IP} we arrive at
\[\|\beta\|^2_{H^{s-1}}\leq C\left([\lambda-\bD(f)[\beta]]^2_{H^{s-1}}+\|\beta\|_2^2\right)
+\frac{1}{2}\|\beta\|_{H^{s-1}}^2.\]
Finally, using Theorem \ref{L2spec} again, we obtain the estimate \eqref{Cdeltas}.
 The isomorphism property of~${\lambda-\bD(f)}$, with $\lambda\in\R$ with~$|\lambda|> 1/2$ and~$f\in H^s(\R)$,  follows by the same continuity argument as in the $L_2$ result. 
\end{proof}

For the $H^2$ result we need an additional estimate for the operators $B_{n,m}$  with higher regularity of the arguments.

\begin{lemma}\label{L:MP1}
 Let  $ n,\, m\in\N$ and  $a_1,\ldots, a_m\in H^2(\R)$   be given. 
  Then, there exists a constant~$C$, depending  only on~$n,\, m$,  and $\max_{1\leq i\leq m}\|a_i\|_{H^2}$, such that
\begin{align} 
\| B_{n,m}(a_1,\ldots, a_{m})[b_1,\ldots, b_n,h]\|_{H^1}\leq C \|h\|_{H^1}\prod_{i=1}^{n}\|b_i\|_{H^2} \label{FER1}
\end{align}
for all $b_1,\ldots, b_n\in H^2(\R)$ and $h\in H^1(\R).$ 

Moreover,  $B_{n,m}\in {\rm C}^{1-}(H^2(\R)^m,\kL^{n}_{\rm sym}(H^2(\R),\kL(H^1(\R)))).$ 
 \end{lemma}
\begin{proof}
We first show that $\varphi:=B_{n,m}(a_1,\ldots, a_{m})[b_1,\ldots, b_n,h]$ belongs to $H^1(\R)$. 
Recalling that the group $\{\tau_\xi\}_{\xi\in\R}\subset\kL(H^r(\R)),$ $r\geq0$,   has generator
$[f\mapsto f']\in\kL(H^{r+1}(\R),H^r(\R)),$
 it suffices to show that  ${D_\xi\varphi:=(\tau_\xi\varphi-\varphi)/\xi}$ converges in $L_2(\R)$ when letting $\xi\to0$. 
In view of~\eqref{diffid} we write
\begin{align*}
 D_\xi\varphi&=\sum_{i=1}^nB_{n,m}(\tau_\xi a_1,\ldots,\tau_\xi a_m)\big[b_1,\ldots,b_{i-1}, D_\xi b_i,\tau_\xi b_{i+1},\ldots,\tau_\xi b_n,\tau_\xi h\big]\\[1ex]
&\hspace{0,45cm}+ B_{n,m}(\tau_\xi a_1,\ldots,\tau_\xi a_m)\big[b_1,\ldots,, b_n, D_\xi h\big]\\[1ex]
&\hspace{0,45cm}-\sum_{i=1}^mB_{n+2,m+1}(\tau_\xi a_1,\ldots,\tau_\xi a_i,a_i,\ldots,a_m)\big[b_1,\ldots,b_n, D_\xi a_i,\tau_\xi a_i+a_i, h\big].
\end{align*}
 Lemma \ref{L:MP0'}  and Lemma \ref{L:MP0}~(i) enable us to pass to the limit $\xi\to0$ in $L_2(\R)$ in  this equality. 
 Hence, $\varphi\in H^1(\R)$ and
\begin{equation}\label{FDER}
\begin{aligned}
\varphi'&=B_{n,m}(  a_1,\ldots,  a_m) [b_1,\ldots , b_n, h' ]\\[1ex]
&\hspace{0,45cm}+\sum_{i=1}^nB_{n,m}(a_1,\ldots,a_m)[b_1,\ldots,b_{i-1}, b_i',b_{i+1},\ldots  b_n, h]\\[1ex]
&\hspace{0,45cm}-2\sum_{i=1}^mB_{n+2,m+1}( a_1,\ldots,  a_i, a_i,\ldots,a_m) [b_1,\ldots,b_n, a_i',a_i, h ].
\end{aligned}
\end{equation}
The estimate \eqref{FER1} is  a consequence of Lemma \ref{L:MP0'}  and Lemma \ref{L:MP0}~(i). 
The local Lipschitz continuity property follows from an  repeated application of \eqref{diffid} and \eqref{FER1}.
\end{proof}
 
As a consequence of Lemma~\ref{L:MP1} and \eqref{FDER} we obtain the following result.

\begin{cor}\label{C:1}
  $B_{n,m}\in {\rm C}^{1-}(H^3(\mathbb{R})^{m},\kL^{n}_{\rm sym}(H^3(\mathbb{R}),\kL( H^2(\mathbb{R})))) $ for all ${n,\, m\in\N}$.
\end{cor}

\begin{thm}\label{T:isomH2}
The operator ${\lambda-\bD(f)\in\kL(H^2(\R)^2)}$ is an isomorphism for all~$f\in H^3(\R)$  and~$\lambda\in\R$ with~${|\lambda|>1/2}$.
\end{thm}
\begin{proof}
Fix \mbox{$f\in H^3(\R)$}.
From \eqref{DFB} and Corollary \ref{C:1} we get  \mbox{$\bD(f)\in\kL(H^2(\R)^2)$.}
Recalling~\eqref{FDER}, we further have
\begin{align}\label{Diff}
(\bD(f)[\beta])''-\bD(f)[\beta'']=T_{\rm lot}[\beta], \qquad \beta\in H^2(\R)^2,
\end{align}
 where each component of $T_{\rm lot}[\beta]$ is a linear combination of terms
\begin{align*}
&B_{n,m}(f,\ldots,f)[f'',f,\ldots ,f, (f')^k\beta_i],&&B_{n,m}(f,\ldots,f)[f',f',f,\ldots ,f, (f')^k\beta_i],\\[1ex]
& B_{n,m}(f,\ldots,f)[f',f,\ldots ,f, ((f')^k\beta_i)'],&&  B_{n,m}(f,\ldots,f)[f,\ldots ,f, f'''\beta_i],
\end{align*}
where $n,\,m\in\N$ satisfy  $0\leq n,\,m\leq 7$  and $k\in\{0,\, 1\}$.
 From Lemma \ref{L:MP0'} and  Lemma~\ref{L:MP0}~(i) (with~$s=7/4$) we conclude that 
\begin{equation}\label{Tlotest}
\|T_{\rm lot}[\beta]\|_2\leq C\|\beta\|_{ H^{1}}, \qquad \beta\in H^2(\R)^2.
\end{equation}
Given $\lambda\in\R$ with $|\lambda|>1/2$,   we pick  $\delta\in(0,1)$   with $|\lambda|\geq 1/2+\delta$ and $\|f'\|_{\infty}\leq 1/\delta.$
Since by  Theorem~ \ref{L2spec} we have $\|(\mu-\bD(f))^{-1}\|_{\kL(L_2(\R)^2)}\leq C$ for all $\mu\in\R$ with $|\mu|\geq 1/2+\delta$,  we deduce from \eqref{Diff}, \eqref{Tlotest}, and \eqref{IP} that
\begin{align*}
\|\beta\|_{H^2}&\leq C(\|\beta''\|_2+\|\beta\|_2)\leq C(\|(\mu-\bD(f))[\beta'']\|_2+\|\beta\|_2)\\[1ex]
&\leq C\big(\|(\mu-\bD(f))[\beta]''\|_2+\|T_{\rm lot}[\beta]\|_2+\|\beta\|_2\big)\leq C\big(\|(\mu-\bD(f))[\beta]''\|_2+ \|\beta\|_{H^{1}}\big)\\[1ex]
&\leq\tfrac{1}{2}\|\beta\|_{H^2}+C\big(\|(\mu-\bD(f))[\beta]''\|_2+\|\beta\|_2\big)\\[1ex]
&\leq\tfrac{1}{2}\|\beta\|_{H^2}+C\big(\|(\mu-\bD(f))[\beta]''\|_2+\|(\mu-\bD(f))[\beta]\|_2\big),
\end{align*}
hence
\[\|\beta\|_{H^2}\leq C\|(\mu-\bD(f))[\beta]\|_{H^2}\]
for all $\beta\in H^2(\R)^2$ and $\mu\in\R$ with $|\mu|\geq 1/2+\delta$.
The result follows now by the same continuity argument as in the proof of Theorem \ref{T:isom2}.
\end{proof}

 %%%%%%%%%%%%%%%%%%%%%%%%%%%%%%%%%%%%%%%%%%%%%%%%%%
%%%%%%%%%%%%%%%%%%%%%%%%%%%%%%%%%%%%%%%%%%%%%%%%%%
%%%%%%%%%%%%%%%%%%%%%%%%%%%%%%%%%%%%%%%%%%%%%%%%%%
%%%%%%%%%%%%%%%%%%%%%%%%%%%%%%%%%%%%%%%%%%%%%%%%%%
\section{The contour integral formulation}\label {Sec:13}
%%%%%%%%%%%%%%%%%%%%%%%%%%%%%%%%%%%%%%%%%%%%%%%%%%
%%%%%%%%%%%%%%%%%%%%%%%%%%%%%%%%%%%%%%%%%%%%%%%%%%
%%%%%%%%%%%%%%%%%%%%%%%%%%%%%%%%%%%%%%%%%%%%%%%%%%
%%%%%%%%%%%%%%%%%%%%%%%%%%%%%%%%%%%%%%%%%%%%%%%%%%
 In this section we formulate the Stokes evolution problem \eqref{STOKES} 
as an nonlinear  evolution problem having only~$f$ as unknown, cf.~\eqref{NNEP}.

Based on the results established in Section~\ref{sec:aux}, Section~\ref{Sec:11}, and Appendix~\ref{neargamma}  we start by proving that for each $f\in H^3(\R)$, the boundary value problem
\begin{equation}\label{probintFT}
\left.
\begin{array}{rclll}
\mu^\pm\Delta v^\pm-\nabla p^\pm&=&0&\mbox{in $\Omega^\pm$,}\\
\vdiv v^\pm&=&0&\mbox{in $\Omega^\pm$,}\\
v^+&=&v^-&\mbox{on $\Gamma$,}\\{}
[T_\mu(v, p)]\tilde\nu&=&-\sigma\tilde\kappa\tilde\nu&\mbox{on $\Gamma$,}\\
(v^\pm, p^\pm)(x)&\to&0&\mbox{for $|x|\to\infty$}
\end{array}\right\}
\end{equation}
has a unique solution $(v,p)\in X_f$ with the property that $ v^\pm|_\G \circ\Xi_f\in H^2(\R)^2$.
 This is established in Proposition~\ref{P:STOsol} below, where we also provide an implicit  formula for~$v^\pm|_\G$ in terms of contour integrals on~$\Gamma$.
 This representation allows to recast the kinematic boundary condition   \eqref{probint}$_6$  in the form~\eqref{NNEP}.

 With the substitution $\tilde v^\pm:=\mu_\pm v^\pm$, Problem  \eqref{probintFT} is equivalent to 
\begin{equation}\label{TSP}
\left.
\begin{array}{rclll}
 \Delta \tilde v^\pm-\nabla p^\pm&=&0&\mbox{in $\Omega^\pm$,}\\
\vdiv \tilde v^\pm&=&0&\mbox{in $\Omega^\pm$,}\\
\mu_- \tilde v^+-\mu_+ \tilde v^-&=&0&\mbox{on $\Gamma$,}\\{}
[T_1(\tilde v, p)]\tilde\nu&=&-\sigma\tilde\kappa\tilde\nu&\mbox{on $\Gamma$,}\\
( \tilde v^\pm, p^\pm)&\to&0&\mbox{for $|x|\to\infty$}.
\end{array}\right\}
\end{equation}
We construct the solution to \eqref{TSP} by splitting
\[(\tilde v,p)=(w_s,q_s)+(w_d,q_d)\]
where $(w_s,q_s),\,(w_d,q_d)\in X_f$ satisfy
\begin{equation}\label{TSPsigma}
\left.
\begin{array}{rclll}
 \Delta w_{s}^\pm-\nabla q_{s}^\pm&=&0&\mbox{in $\Omega^\pm$,}\\
\vdiv w_{s}^\pm&=&0&\mbox{in $\Omega^\pm$,}\\
 w_{s}^+-w_{s}^-&=& 0&\mbox{on $\Gamma$,}\\{}
[T_1(w_{s},q_{s})]\tilde\nu&=&-\sigma\tilde\kappa\tilde\nu&\mbox{on $\Gamma$,}\\
(w_{s}^\pm,q_{s}^\pm)&\to&0&\mbox{for $|x|\to\infty$}
\end{array}\right\}
\end{equation}
and
\begin{equation}\label{TSPbeta}
\left.
\begin{array}{rclll}
 \Delta w_{d}^\pm-\nabla q_{d}^\pm&=&0&\mbox{in $\Omega^\pm$,}\\
\vdiv w_{d}^\pm&=&0&\mbox{in $\Omega^\pm$,}\\
\mu_-w_d^+-\mu_+w_d^-&=&(\mu_+-\mu_-)w_s&\mbox{on $\Gamma$,}\\
{}[T_1(w_{d},q_{d})]\tilde\nu&=&0&\mbox{on $\Gamma$,}\\
(w_{d}^\pm,q_{d}^\pm)&\to&0&\mbox{for $|x|\to\infty$.}
\end{array}\right\}
\end{equation}
The system \eqref{TSPsigma} has been  studied in \cite{MP2021}.
According to \cite[Theorem 2.1 and  Remark~A.2]{MP2021},  there exists   precisely one  solution~${(w_{s} ,q_{s}):=(w_{s}(f),q_{s}(f))\in X_f}$ to \eqref{TSPsigma}. 
It satisfies
\[
w_{s} \in {\rm C}^\infty(\R^2\setminus\Gamma)\cap {\rm C}^1(\R^2)\quad\text{and}\quad q_{s}^\pm\in {\rm C}^\infty(\Omega^\pm)\cap {\rm C}(\overline{\Omega^\pm}).
\] 
Moreover, recalling  \eqref{defB0} and  \cite[Eqns. (2.2), (2.3),  (A.2)]{MP2021}, the trace $w_{s}(f)|_\Gamma$  can be expressed  via
\begin{equation}\label{defWf}
w_s(f)|_\Gamma\circ\Xi=:G(f):=(G_1(f),G_2(f)),
\end{equation} 
with 
\begin{equation}\label{vonGamma}
\begin{aligned}
 4\pi\sigma^{-1} G_1(f)&:= (B_{0,2}^0(f)-B_{2,2}^0(f))[\phi_1(f)+f'\phi_2(f)]  \\[1ex]
 &\hspace{0.55cm}+ B_{1,2}^0(f)[3f'\phi_1(f)-\phi_2(f)] +B_{3,2}^0(f)[f'\phi_1(f)+\phi_2(f)] ,\\[1ex]
 4\pi\sigma^{-1} G_2(f)&:= (B_{1,2}^0(f)-B_{3,2}^0(f))[\phi_1(f)+f'\phi_2(f)] \\[1ex]
   &\hspace{0.55cm} -B_{0,2}^0(f)[f'\phi_1(f)+\phi_2(f)]+B_{2,2}^0(f)[f'\phi_1(f)-3\phi_2(f)] ,
\end{aligned}
\end{equation}
where $ \phi_i(f)\in H^2(\R),$ $i\in\{1,\,2\}$, are given by
\begin{align}\label{Phii}
 \phi_1(f):=\frac{{f'}^2}{\omega+\omega^2}\qquad\text{and}\qquad \phi_2(f):=\frac{f'}{\omega}.
 \end{align}
We point out that  Corollary \ref{C:1} yields $G_i(f)\in H^2(\R),$ $i\in\{1,\,2\}$.

It remains to show that the boundary value problem~\eqref{TSPbeta} has a unique solution~${(w_d,q_d)\in X_f}$ with $w_d^\pm|_\G\circ\Xi\in H^2(\R)^2.$
To prove the existence, we solve in $X_f$, for  given $\beta\in H^2(\R)^2,$  the auxiliary problem  \eqref{bvpaux} and denote its solution by~$(w,q)=(w,q)[\beta]$.
In view of Lemma~\ref{l:neargamma}~(i)  we have
\begin{align*} 
(\mu_-w^+-\mu_+w^-)|_\G\circ\Xi=(\mu_++\mu_-)\Big(\frac{1}{2}+a_\mu\bD(f)\Big)[\beta].
\end{align*} 
Therefore  $(w_d,q_d):=(w,q)[\beta]$ solves~\eqref{TSPbeta} if and only if
\begin{align}\label{TOSO}
\Big(\frac{1}{2}+a_\mu\bD(f)\Big)[\beta]=a_\mu G(f),
\end{align} 
 where 
\[a_\mu:=\frac{\mu_+-\mu_-}{\mu_++\mu_-}\in(-1,1).\]
Theorem~\ref{T:isomH2} implies that  equation~\eqref{TOSO} has a unique solution $\beta=:\beta(f)\in H^2(\R)^2$.
This establishes not only the existence but also the uniqueness of the solution to~\eqref{TSPbeta}. \pagebreak

 Summarizing, we  have shown the following result:
\begin{prop}\label{P:STOsol}
 Given $f\in H^3(\R)$, the  boundary value problem \eqref{probintFT} has a  unique solution~${(v,p)\in X_f}$ such that $ v^\pm|_\G \circ\Xi\in H^2(\R)^2.$
Moreover, 
\begin{align*}
v^\pm|_\G\circ\Xi=\frac{G(f)}{\mu_\pm}+\frac{1}{\mu_\pm}\left(-\bD(f)\pm \frac{1}{2}\right)[\beta(f)],
\end{align*}
where $G(f)\in H^2(\R)^2$ is defined in \eqref{defWf}-\eqref{vonGamma} and $\beta(f)\in H^2(\R)^2$ is the unique solution to~\eqref{TOSO}.
\end{prop}

 From this result and \eqref{STOKES} we infer, under the assumption that $\G(t)$ is at each time instant~${t\geq0}$ the graph of a 
 function $f(t)\in H^3(\R)$ and that $(v(t),p(t))$ belongs to~$X_{f(t)}$  and satisfies~$v(t)^\pm|_{\Gamma(t)}\circ\Xi(f(t))\in H^2(\R)^2,$ that \eqref{probint} can be recast as
\be\label{evol0}
\p_tf=\frac{1}{\mu_+}\Big\langle G(f)-\bD(f)[\beta(f)]+ \frac{1}{2}\beta(f)\,\Big|\,(-f',1)^\top\Big\rangle
=\frac{1}{\mu_+-\mu_-}\big\langle\beta(f)\,|\,(-f',1)^\top\big\rangle.
\ee
Here $\langle \cdot\,|\,\cdot\rangle$ denotes the scalar product on $\R^2$.

 Using the results in Section~\ref{Sec:11} and \cite{MP2021} we can formulate the latter equation as an evolution equation in $H^{s-1}(\R)^2$, where $s\in(3/2,2)$ is fixed in the remaining.
To this end we first infer from \cite[Corollary C.5]{MP2021} that, given~$n,\, m\in\N$, we have
\begin{align}\label{Bsmooth}
[f\mapsto B^0_{n,m}(f)]\in {\rm C}^\infty(H^s(\R), \kL( H^{s-1}(\mathbb{R}))).
\end{align}
Further, \cite[Lemma 3.5]{MP2021} ensures for the mappings defined in \eqref{Phii} that  
\begin{align}\label{regphii}
[f\mapsto\phi_i(f)] \in {\rm C}^\infty(H^s(\mathbb{R}), H^{s-1}(\mathbb{R})),\qquad i=1,\, 2.
\end{align}
Additionally, for any~${f_0\in H^s(\mathbb{R}),}$ the Fr\'echet derivative $\p\phi_i(f_0)$  is given by 
$$\p\phi_i(f_0)=a_i(f_0)\frac{d}{dx},\qquad i=1,\, 2,$$
with
\be\label{defa12}
a_1(f_0):=\frac{f_0'(2+f_0'^2+2\sqrt{1+f_0'^2})}{\sqrt{1+f_0'^2}(\sqrt{1+f_0'^2}+1+f_0'^2)^2}\qquad\text{and}\qquad a_2(f_0):=\frac{1}{(1+f_0'^2)^{3/2}}.
\ee
  It is easy to check, by arguing as in 
  \cite[Lemma~C.1]{MP2021},
   that $\phi_i$, $i=1,\, 2$, maps bounded sets in~$H^s(\R)$ to bounded sets in~$H^{s-1}(\R)$.
 This observation, the relations  \eqref{vonGamma}, \eqref{Bsmooth}, \eqref{regphii},   and Lemma \ref{L:MP0} combined enable us to conclude  that  the map
   defined in \eqref{defWf}-\eqref{vonGamma}   satisfies
\begin{align}\label{regWf}
[f\mapsto G(f)]\in {\rm C}^\infty(H^{s}(\R),H^{s-1}(\R)^2),
\end{align}
 and also that  $G$   maps bounded sets in~$H^s(\R)$ to bounded sets in~$H^{s-1}(\R)^2$.

Moreover, recalling \eqref{DFB}, we infer from \eqref{Bsmooth} that
\begin{align}\label{regbD}
\bD\in {\rm C}^\infty(H^s(\R),\kL(H^{s-1}(\R)^2)).
\end{align}
In view of~\eqref{regWf} and of Theorem~\ref{T:isom2} we can solve, for given $f\in H^s(\R)$, the equation~\eqref{TOSO} in $H^{s-1}(\R)^2$. 
Its  unique solution is given by
\begin{equation}\label{betaf}
\beta(f):=2a_\mu(1+2a_{\mu}\bD(f))^{-1}[G(f)]\in H^{s-1}(\R)^2,
\end{equation}
and, since the mapping which associates to an isomorphism its inverse is smooth, we obtain from Theorem~\ref{T:isom2}, \eqref{regWf}, and \eqref{regbD}  that 
 \begin{align}\label{regBeta}
 \big[f\mapsto\beta(f)]\big]\in {\rm C}^\infty(H^{s}(\R),H^{s-1}(\R)^2).
 \end{align}
  Furthermore,  \eqref{betaf} and the estimate \eqref{Cdeltas} imply that  $\beta$  inherits from $G$ the property to map bounded sets in~$H^s(\R)$ to bounded sets in~$H^{s-1}(\R)^2$.
Summarizing, in a compact form, the Stokes flow  problem \eqref{STOKES} can be recast as the evolution problem
\begin{align}\label{NNEP}
\frac{df}{dt}=\Phi(f(t)),\quad t\geq0,\qquad f(0)=f_0,
\end{align}
where $\Phi:H^{s}(\R)\to H^{s-1}(\R)$ is defined, cf. \eqref{evol0},  by
\begin{align}\label{PHI}
\Phi(f):=\frac{1}{\mu_+-\mu_-}\langle\beta(f)|(-f',1)^\top\rangle.
\end{align}
Observe that, due to \eqref{regBeta}, 
\begin{align}\label{REGphi}
\Phi\in {\rm C}^\infty(H^s(\mathbb{R}), H^{s-1}(\mathbb{R})),
\end{align}
 and that $\Phi$  maps bounded sets in~$H^s(\R)$ to bounded sets in~$H^{s-1}(\R).$

\section{\label{Sec:linfinal} Linearization, localization, and proof of the main result}

 We are going to prove that the nonlinear and nonlocal problem \eqref{NNEP} is parabolic in~$H^s(\R)$ 
 in the sense that the Fr\'echet derivative $\p\Phi(f_0)$, 
generates an analytic semigroup in $\kL(H^{s-1}(\R))$ for each~${f_0\in H^s(\R)}$.
  This property then enables us to use the abstract  existence results from \cite{L95} in the proof of our  main result Theorem~\ref{MT1}.

 \begin{thm}\label{T:GP}
 For any $f_0\in H^s(\mathbb{R})$, the Fr\'echet derivative $\p\Phi(f_0)$,  considered as an unbounded operator in $H^{s-1}(\R)$ with dense domain $H^{s}(\R)$,  
 generates an analytic semigroup in $\kL(H^{s-1}(\R))$.
\end{thm}

 The proof of Theorem \ref{T:GP}  requires some preparation.
To start,  fix $f_0\in H^s(\R)$, $s'\in(3/2,s)$, and let
$\beta_0:=\beta(f_0):=(\beta_{0}^1,\beta_{0}^2)^\top$.  We have $\beta_0\in H^{s-1}(\R)^2.$

Differentiating the relations  \eqref{PHI} and \eqref{betaf}, we get
\begin{align}\label{pPhi}
\p\Phi(f_0)[f]= \frac{1}{\mu_+-\mu_-}\langle \p\beta(f_0)[f]|(-f_0',1)^\top\rangle-\frac{\beta_0^1f'}{\mu_+-\mu_-}
\end{align}
and 
\begin{align}\label{pbeta}
(1+2a_{\mu}\bD(f_0))[\p\beta(f_0)[f]]=2a_\mu\p G(f_0)[f]-2a_\mu\p\bD(f_0)[f][\beta_0].
\end{align}

For the computation of  $\p\bD(f_0)[f][\beta_0]$  and $\p G(f_0)[f]$ we use  the relation
\[\partial B_{n,2}^0(f_0)[f][h]=n B_{n,2}(f_0,f_0)[f,f_0,\ldots f_0,h]-4B_{n+2,3}(f_0,f_0,f_0)[f,f_0,\ldots,f_0,h],\quad n\in\N,\]
see \cite[Lemma C.4]{MP2021}. 
Additionally we use Lemma \ref{L:MP0}~(iii) to rewrite this expression as
\begin{align*}
\partial B_{n,2}^0(f_0)[f][h]&=h\big(
nB^0_{n-1,2}(f_0)[f']-4B^0_{n+1,3}(f_0)[f']\big)+R_{1,n}[f,h]\\
&=h\big(n B^0_{n-1,3}(f_0)[f']+(n-4)B^0_{n+1,3}(f_0)[f']\big)+R_{1,n}[f,h],
\end{align*}
where  $n B^0_{n-1,3}(f_0):=0$ for $n=0$ and
\[\|R_{1,n}[f,h]\|_{H^{s-1}}\leq C\|h\|_{H^{s-1}}\|f\|_{H^{s'}}, \]
 with a constant $C$  independent of $f\in H^{s}(\R)$ and $h\in H^{s-1}(\R).$
Using these relations, we infer from \eqref{DFB} that
\begin{align}\label{pDE1}
(\p\bD(f_0)[f][\beta_0])_i
&=\frac{1}{\pi}\big\{B^0_{i+k-2,2}[f'\beta_0^k]
+\beta_0^k\big((i+k-2)f_0'B^0_{i+k-3,3}+(i+k-6)f_0'B^0_{i+k-1,3}\nonumber\\
&\hspace{1cm}-(i+k-1)B^0_{i+k-2,3}-(i+k-5)B^0_{i+k,3}\big)[f']\big\}+R_{2,i}[f]
\end{align}
for $i=1,\,2$, where we used the shorthand notation $B_{n,m}^0:=B_{n,m}^0(f_0)$  and
\begin{align}\label{pDE3}
\|R_{2,i}[f]\|_{H^{s-1}}\leq  C\|f\|_{H^{s'}},\quad f\in H^s(\R).
\end{align}
Taking the derivative of \eqref{vonGamma}, the same arguments yield 
\begin{equation}\label{pPsi1}
\begin{aligned}
  4\pi\sigma^{-1}\p G_i(f_0)[f]&=T_{i,1}(f_0)[f]+T_{i,2}(f_0)[f]+R_{3,i}[f],\qquad i=1,\,2,
\end{aligned}
\end{equation}
where
\begin{equation}\label{pPsi2}
\begin{aligned}
T_{1,1}(f_0)[f]:=&(B^0_{0,2}-B^0_{2,2})[(a_1+\phi_2+f_0'a_2)f']+B^0_{1,2}[(3(\phi_1+f_0'a_1)-a_2)f']\\[1ex]
&+B_{3,2}^0[(\phi_1+f_0'a_1+a_2)f'],\\[1ex]
T_{1,2}(f_0)[f]:=&\phi_1(3f_0'B_{0,3}^0-6B_{1,3}^0-6f_0'B_{2,3}^0+2B_{3,3}^0-f_0'B_{4,3}^0)[f']\\[1ex]
&+\phi_2(-B_{0,3}^0-6f_0'B_{1,3}^0+6B_{2,3}^0+2f_0'B_{3,3}^0-B_{4,3}^0)[f'],\\[1ex]
T_{2,1}(f_0)[f]:=&-B_{0,2}^0[(\phi_1+f_0'a_1+a_2)f']+(B^0_{1,2}-B^0_{3,2})[(a_1+\phi_2+f_0'a_2)f']\\[1ex]
&+B^0_{2,2}[(\phi_1+f'_0a_1-3a_2)f'],\\[1ex]
T_{2,2}(f_0)[f]:=&\phi_1(B_{0,3}^0+6f_0'B_{1,3}^0-6B_{2,3}^0-2f_0'B_{3,3}^0+B_{4,3}^0)[f']\\[1ex]
&+\phi_2(f_0'B_{0,3}^0-2B_{1,3}^0-6f_0'B_{2,3}^0+6B_{3,3}^0+f_0'B_{4,3}^0)[f'],
\end{aligned}
\end{equation}
cf. \cite[Eq. (3.7)-(3.9)]{MP2021}.
Here we used the shortened notation $a_i:=a_i(f_0)$ and $\phi_i:=\phi_i(f_0)$ for $i=1,\,2$ and
\begin{align}\label{pPsi3}
\|R_{3,i}[f]\|_{H^{s-1}}\leq  C\|f\|_{H^{s'}},\quad f\in H^s(\R).
\end{align}

In order to prove Theorem~\ref{T:GP} we  consider the path $\Psi:[0,1]\longrightarrow\kL(H^{s}(\mathbb{R}), H^{s-1}(\mathbb{R}))$ defined by
\begin{align}\label{PSI}
\Psi(\tau)[f]:= \frac{1}{\mu_+-\mu_-}\langle \kB(\tau)[f]|(-\tau f_0',1)^\top\rangle-\frac{\tau\beta_0^1 f'}{\mu_+-\mu_-} 
\end{align}
for  $\tau\in[0,1]$ and $f\in H^s(\R),$
where $\kB(\tau)[f]$ is defined by
\begin{align}\label{betatau}
(1+2\tau a_{\mu}\bD(f_0))[\kB(\tau)[f]]=2a_\mu(\p G(\tau f_0)[f]-\tau \p\bD(f_0)[f][\beta_0]).
\end{align}
Theorem \ref{T:isom2},   \eqref{pDE1}--\eqref{pPsi3}, and Lemma \ref{L:MP0} (ii) ensure that~${\kB:[0,1]\longrightarrow\kL\big(H^s(\RRM),H^{s-1}(\RRM)^2\big)}$ is well-defined, and
\begin{equation}\label{normB}
\|\kB(\tau)[f]\|_{H^{s-1}}\leq C\|f\|_{H^s},\qquad \tau\in[0,1], \, f\in H^s(\R),
\end{equation}
with $C$ independent of $f$ and $\tau$.
We also note  that  both paths $\kB$ and $\Psi$ are continuous and~${\Psi(1)=\p\Phi(f_0)}$.
Besides, since 
$$\kB(0)=2a_\mu \p G(0)=\Big(0,-\frac{ 2a_\mu\sigma}{4} H\circ \frac{d}{d \xi}\Big)^\top,$$
where  $H =\pi^{-1}B_{0,0}$ is the Hilbert transform, we  observe that $\Psi(0)$ is the Fourier multiplier 
\be\label{FMPsi0}
\Psi(0)=-\frac{\sigma }{2(\mu_++\mu_-)} H\circ \frac{d}{d\xi}=-\frac{\sigma }{2(\mu_++\mu_-)}\Big(- \frac{d^2}{d\xi^2}\Big)^{1/2}.
\ee

We next locally approximate the operator  $\Psi(\tau)$,  $\tau\in[0,1]$,   by  certain Fourier multipliers~$\bA_{j,\tau}$, cf. Theorem~\ref{T:AP} below.
For this purpose, given $\e\in(0,1)$, we choose $N=N(\e)\in\N$ and a so-called finite~$\e$-localization family, that is a set 
\[\{(\pi_j^\e,\xi_j^\e)\,|\, -N+1\leq j\leq N\}\]
 such that
\begin{align*}
\bullet\,\,\,\, \,\,&\text{$\pi_j^\e\in {\rm C}^\infty(\mathbb{R},[0,1]),$ $-N+1\leq j\leq N$, and $\sum_{j=-N+1}^N(\pi_j^\e)^2=1;$}\\[1ex]
\bullet\,\,\,\, \,\,  & \text{$ \supp \pi_j^\e $ is an interval of length $\e$ for all $|j|\leq N-1$  and $ \supp \pi_{N}^\e\subset\{|\xi|\geq 1/\e\}$;} \\[1ex]
\bullet\,\,\,\, \,\, &\text{ $ \pi_j^\e\cdot  \pi_l^\e=0$ if $[|j-l|\geq2, \max\{|j|, |l|\}\leq N-1]$ or $[|l|\leq N-2, j=N];$} \\[1ex]
 \bullet\,\,\,\, \,\, &\text{$\|(\pi_j^\e)^{(k)}\|_\infty\leq C\e^{-k}$ for all $ k\in\N, -N+1\leq j\leq N$;} \\[1ex]
 \bullet\,\,\,\, \,\, &\xi_j^\e\in\supp\pi_j^\e,\; |j|\leq N-1. 
 \end{align*} 
 The real number $\xi_N^\e$ plays no role in the analysis below.
 To each  $\e$-localization family  we associate  a norm on $H^r(\mathbb{R}),$ $r\geq 0$, which is  equivalent to the standard  norm.
Indeed,   given~${r\geq0}$ and~$\e\in(0,1)$ , there exists a constant $c=c(\e,r)\in(0,1)$ such that
\begin{align}\label{EQNO}
c\|f\|_{H^r}\leq \sum_{j=-N+1}^N\|\pi_j^\e f\|_{H^r}\leq c^{-1}\|f\|_{H^r},\qquad f\in H^r(\mathbb{R}).
\end{align}
 
To introduce the aforementioned Fourier multipliers~$\bA_{j,\tau}$,  we  first define the coefficient functions $\alpha_\tau,\, \beta_\tau:\R\longrightarrow\R$, $\tau\in[0,1]$,  by the relations
 \begin{equation}\label{defalphabeta}
 \alpha_\tau:=\frac{\sigma}{2(\mu_++\mu_-)}\big(a_2(\tau f_0)+\tau f_0'a_1(\tau f_0)\big), \qquad  \beta_\tau:= -\frac{\tau\beta_0^1 }{ \mu_+ -\mu_-} .  
 \end{equation}
 We now set 
 \begin{alignat}{2}\label{defAjtau}
\bA_{j,\tau }&:=\bA_{j,\tau}^\e&&:=- \alpha_\tau(\xi_j^\e) \Big(-\frac{d^2}{d\xi^2}\Big)^{1/2}+\beta_\tau (\xi_j^\e)\frac{d}{d\xi}, \quad |j|\leq N-1,\nonumber\\
      \bA_{N,\tau }&:=\bA_{N,\tau }^\e&&:= -  \frac{\sigma}{2(\mu_++\mu_-)} \Big(-\frac{d^2}{d\xi^2}\Big)^{1/2}.
 \end{alignat}
 We obviously have
 \[\bA_{j,\tau}\in\kL(H^s(\mathbb{R}), H^{s-1}(\mathbb{R})), \qquad\text{$-N+1\leq j\leq N$, $\tau\in[0,1]$.} \]
 
The  following estimate of the localization error is the main step in the proof of Theorem~\ref{T:GP}. 

\begin{thm}\label{T:AP} 
Let $\mu>0$ be given and fix $s'\in (3/2,s)$. 
Then there exist $\e\in(0,1)$ and  a constant $K=K(\e)$ such that 
 \begin{equation}\label{D1}
  \|\pi_j^\e \Psi(\tau) [f]-\bA_{j,\tau}[\pi^\e_j f]\|_{H^{s-1}}\leq \mu \|\pi_j^\e f\|_{H^s}+K\|  f\|_{H^{s'}}
 \end{equation}
 for all $-N+1\leq j\leq N$, $\tau\in[0,1],$  and  $f\in H^s(\mathbb{R})$. 
\end{thm}

Before proving Theorem~\ref{T:AP} we first present some auxiliary lemmas which are used in the proof.
We start with an estimate for the commutator  $[B_{n,m}^0(f),\varphi]$ (we will apply this estimate in the particular case $\varphi=\pi_j^\e$, $-N+1\leq j\leq N$).
\smallskip

\begin{lemma}\label{L:AL1} 
Let $n,\, m \in \N$,  $s\in(3/2, 2)$, $f\in H^s(\R)$, and  ${\varphi\in {\rm C}^1(\R)}$ with uniformly continuous derivative $\varphi'$ be given. 
Then, there exist  a constant $K$ that depends only on $ n,$ $m, $ $\|\varphi'\|_\infty, $ and $\|f\|_{H^s}$  such that 
 \begin{equation}\label{LB2}
  \|\varphi B_{n,m}(f,\ldots,f)[f,\ldots,f, h]- B_{n,m}(f,\ldots,f)[f,\ldots,f, \varphi h]\|_{H^{1}}\leq K\| h\|_{2}
 \end{equation}
for all   $h\in L_2(\R)$.
\end{lemma}
\begin{proof}
This result is a particular case of \cite[Lemma 12]{AM21x}.
\end{proof}

The results in Lemma~\ref{L:AL2}-Lemma~\ref{L:AL6} below describe how  to ``freeze the coefficients'' of the multilinear operators $B_{n,m}^0.$  
 For these operators, this technique has been first developed in \cite{MBV19} in the study of the Muskat problem.
\begin{lemma}\label{L:AL2} 
Let $n,\, m \in \N$, $3/2<s'<s<2$, and  $\nu\in(0,\infty)$ be given. 
Let further~${f\in H^s(\mathbb{R})}$ and  $ \oo\in \{1\}\cup H^{s-1}(\mathbb{R})$.
For any sufficiently small $\e\in(0,1)$, there is
a constant $K$ depending only on $\e,\, n,\, m,\, \|f\|_{H^s},$ and $\|\oo\|_{H^{s-1}}$ (if $\oo\neq1$)  such that 
 \begin{equation*} 
  \Big\|\pi_j^\e\oo B_{n,m}^0(f)[ h]-\frac{\oo(\xi_j^\e)(f'(\xi_j^\e))^n}{[1+(f'(\xi_j^\e))^2]^m}B_{0,0}[\pi_j^\e h]\Big\|_{H^{s-1}}\leq \nu \|\pi_j^\e h\|_{H^{s-1}}+K\| h\|_{H^{s'-1}} 
 \end{equation*}
for all $|j|\leq N-1$ and  $h\in H^{s-1}(\mathbb{R})$.
\end{lemma}  
\begin{proof}
See \cite[Lemma~13]{AM21x}. 
\end{proof}

We now provide a similar result as in Lemma~\ref{L:AL2}, the difference to the latter  being that the linear argument of $B_{n,m}$ is now multiplied by a function $a$ that also needs to be frozen at~$\xi_j^\e$.

\begin{lemma}\label{L:AL3} 
Let $n,\, m \in \N$, $3/2<s'<s<2$, and  $\nu\in(0,\infty)$ be given. 
Let further~${f\in H^s(\mathbb{R})}$, $a\in  H^{s-1}(\mathbb{R})$, and $\oo\in \{1\}\cup H^{s-1}(\mathbb{R})$.
For any sufficiently small $\e\in(0,1)$, there is
a constant $K$ only depending on $\e,$ $ n,$ $ m,$ $ \|f\|_{H^s},$  $\|a\|_{H^{s-1}},$  and $ \|\oo\|_{H^{s-1}}$ (if $\oo\neq1$)  such that 
 \begin{equation*} 
\Big\|\pi_j^\e\oo B_{n,m}^0(f)[ ah]-\frac{a(\xi_j^\e)\oo(\xi_j^\e)(f'(\xi_j^\e))^n}{[1+(f'(\xi_j^\e))^2]^m}B_{0,0}[\pi_j^\e h]\Big\|_{H^{s-1}}\leq \nu \|\pi_j^\e h\|_{H^{s-1}}+K\| h\|_{H^{s'-1}}
 \end{equation*}
for all $|j|\leq N-1$ and  $h\in H^{s-1}(\mathbb{R})$.
\end{lemma}  
\begin{proof}
See \cite[Lemma~D.5]{MP2021}.
\end{proof}

Lemma~\ref{L:AL4} and Lemma \ref{L:AL5} are the analogues of Lemma \ref{L:AL2}  corresponding to the case~$j=N$.

\begin{lemma}\label{L:AL4} 
Let $n,\, m \in \N$,  $3/2<s'<s<2$, and  $\nu\in(0,\infty)$ be given. 
Let further~${f\in H^s(\mathbb{R})}$ and  $\oo\in  H^{s-1}(\mathbb{R})$.
For any sufficiently small  $\e\in(0,1)$, there is a constant~$K$ depending only on $\e,\, n,\, m,\, \|f\|_{H^s},$ and $\|\oo\|_{H^{s-1}}$ such that 
  \begin{equation*}
  \|\pi_N^\e\oo B_{n,m}^0(f)[h]\|_{H^{s-1}}\leq \nu \|\pi_N^\e h\|_{H^{s-1}}+K\| h\|_{H^{s'-1}}
 \end{equation*} 
 for $h\in H^{s-1}(\mathbb{R})$.
\end{lemma}  
\begin{proof}
See  \cite[Lemma~14]{AM21x}. 
\end{proof}

Lemma \ref{L:AL5} is the counterpart of Lemma~\ref{L:AL4} in the case when $\oo=1$.

\begin{lemma}\label{L:AL5} 
Let $n,\, m \in \N$,  $3/2<s'<s<2$, and  $\nu\in(0,\infty)$ be given. 
Let further~${f\in H^s(\mathbb{R})}$.
For any sufficiently small  $\e\in(0,1)$, 
there is a constant~$K$ depending only on $\e,\, n,\, m,$ and $ \|f\|_{H^s}$  such that 
 \begin{equation*} 
  \|\pi_N^\e B_{0,m}^0(f)[ h]-B_{0,0}[\pi_N^\e h]\|_{H^{s-1}}\leq \nu \|\pi_N^\e h\|_{H^{s-1}}+K\| h\|_{H^{s'-1}}
 \end{equation*}
 and 
  \begin{equation*}
  \|\pi_N^\e B_{n,m}^0(f)[h]\|_{H^{s-1}}\leq \nu \|\pi_N^\e h\|_{H^{s-1}}+K\| h\|_{H^{s'-1}},\qquad n\geq 1,
 \end{equation*}
 for  all $h\in H^{s-1}(\mathbb{R})$.
\end{lemma}
\begin{proof}
See  \cite[Lemma~15]{AM21x}.
\end{proof}

Finally, Lemma~\ref{L:AL6} below is the analogue of Lemma~\ref{L:AL3} corresponding to the case~${j=N}$.
 
\begin{lemma}\label{L:AL6} 
Let $n,\, m \in \N$, $3/2<s'<s<2$, and  $\nu\in(0,\infty)$ be given. 
Let further~${f\in H^s(\mathbb{R})}$, $a\in H^{s-1}(\mathbb{R})$, and $\oo\in \{1\}\cup H^{s-1}(\mathbb{R})$.
For any sufficiently small $\e\in(0,1)$, there is
a constant $K$ depending on $\e,$ $ n,$ $ m,$ $ \|f\|_{H^s},$  $\|a\|_{H^{s-1}},$ and $ \|\oo\|_{H^{s-1}}$ (if $\oo\neq1$)  such that 
 \begin{equation*} 
\begin{aligned}
  &\|\pi_N^\e\oo B_{n,m}^0(f)[ ah]\|_{H^{s-1}}\leq \nu \|\pi_N^\e h\|_{H^{s-1}}+K\| h\|_{H^{s'-1}}
\end{aligned}  
 \end{equation*}
for  all $h\in H^{s-1}(\mathbb{R})$.
\end{lemma}  
\begin{proof}
See \cite[Lemma~D.6]{MP2021}.
\end{proof}

We are now in a position to prove Theorem~\ref{T:AP}.

\begin{proof}[Proof of Theorem~\ref{T:AP}] 
 Fix $\mu>0$ and let $\e\in(0,1)$. Let further $\{(\pi_j^\e,\xi_j^\e)\,|\, -N+1\leq j\leq N\}$  be a finite $\e$-localization family. 
We choose a second family  $\{\chi_j^\e\,|\, -N+1\leq j\leq N\}$ with the following properties:
\begin{align*}
\bullet\,\,\,\, \,\,  &\text{$\chi_j^\e\in{\rm C}^\infty(\mathbb{R},[0,1])$ and $\chi_j^\e=1$ on $\supp \pi_j^\e$, $-N+1\leq j\leq N$;} \\[1ex]
\bullet\,\,\,\, \,\,  &\text{$\supp \chi_j^\e$ is an interval  of length $3\e$, $|j|\leq N-1$, and $\supp\chi_N^\e\subset \{|x|\geq 1/\e-\e\}$.}  
\end{align*} 

In the arguments that follow we repeatedly use the  estimate
\begin{align}\label{MES}
\|gh\|_{H^{s-1}}\leq  C(\|g\|_\infty\|h\|_{H^{s-1}}+\|h\|_\infty\|g\|_{H^{s-1}})
\end{align} 
which holds for $g,\, h\in H^{s-1}(\mathbb{R})$ and $s\in(3/2,2)$,  with a constant $C$ independent of $g$ and $h$.

Below we denote by~$C$ constants that do not depend on~$\e$ and by~$K$ constants that may  depend on~$\e$.
We need to approximate the linear operators~${\big[f\mapsto\kB_2(\tau)[f]-\tau f_0'\kB_1(\tau)[f]\big]}$ and~${[f\mapsto \beta_0^1f']}$, see~\eqref{PSI}-\eqref{betatau},
where we set~$\kB(\tau)=:(\kB_1(\tau),\kB_2(\tau))^\top$.                                  
The proof  is divided in several steps.
\medskip

\noindent{\em Step 1.}  We  consider the operator $[f\mapsto \beta_0^1f']$.
 Since $\chi_j^\e\pi_j^\e=\pi_j^\e$,~\eqref{MES}  yields
 \begin{align*}
     \|\pi_j^\e(\beta_0^1 f')-\beta_0^1(\xi_j^\e)(\pi_j^\e f)'\|_{H^{s-1}}&\leq C\|\chi_j^\e(\beta_0^1 -\beta_0^1(\xi_j^\e))\|_\infty\|(\pi_j^\e f)'\|_{H^{s-1}}+K\|f\|_{H^{ s'}}
 \end{align*}
 for $|j|\leq N-1$ and
 \begin{align*}
    \|\pi_N^\e(\beta_0^1 f')\|_{H^{s-1}}&\leq C\|\chi_N^\e\beta_0^1\|_\infty\|(\pi_N^\e f)'\|_{H^{s-1}}+K\|f\|_{H^{ s'}}. 
 \end{align*}
 From \eqref{regBeta} we have $\beta_0^1\in{\rm C}^{s-3/2}(\R)$ and $\beta_0^1(\xi)\to 0$ for $|\xi|\to\infty$. 
 Hence, if $\e$ is sufficiently small, then
 \begin{equation}\label{aa1}
 \begin{aligned}
 &\|\pi_j^\e(\beta_0^1 f')-\beta_0^1(\xi_j^\e)(\pi_j^\e f)'\|_{H^{s-1}}\leq \frac{\mu|\mu_+-\mu_-|}{3}\|\pi_j^\e f\|_{H^s}+K\|f\|_{H^{s'}},\qquad |j|\leq N-1,\\[1ex]
 &\|\pi_N^\e(\beta_0^1 f')\|_{H^{s-1}}\leq \frac{\mu|\mu_+-\mu_-|}{3}\|\pi_N^\e f\|_{H^s}+K\|f\|_{H^{s'}}.
 \end{aligned}
 \end{equation}

The approximation procedure for~$\big[f\mapsto\kB_2(\tau)[f]-\tau f_0'\kB_1(\tau)[f]\big]$ is more involved. \medskip

\noindent{\em Step 2.} We prove there exists a constant $C_\kB$ such that
\begin{align}\label{COM}
\|\pi_j^\e\kB(\tau)[f]\|_{H^{s-1}}\leq C_\kB\|\pi_j^\e f\|_{H^s}+K\|f\|_{H^{s'}}
\end{align}
for all $-N+1\leq j\leq N$, $\tau\in[0,1],$ and $f\in H^s(\R)$.
To start, we infer from \eqref{betatau} that 
\begin{equation}\label{COM10}
\begin{aligned}
(1+2\tau a_{\mu}\bD(f_0))[\pi_j^\e\kB(\tau)[f]]&=2a_\mu\pi_j^\e\p G(\tau f_0)[f]-2\tau a_\mu\pi_j^\e\p\bD(f_0)[f][\beta_0]\\[1ex]
&\hspace{0,45cm}+2\tau a_{\mu}\big(\bD(f_0)[\pi_j^\e\kB(\tau)[f]]-\pi_j^\e\bD(f_0)[\kB(\tau)[f]]\big).
\end{aligned}
\end{equation}
To estimate the terms on the right, we use the representations and estimates \eqref{pDE1}--\eqref{pPsi3} together with the commutator estimate from Lemma \ref{L:AL1} 
and the $H^{s-1}$-estimate for the operators~$B_{m,n}$ provided in Lemma \ref{L:MP0}~(ii).
 So we get
\begin{align}\label{COM12}
\|\pi_j^\e\p G(\tau f_0)[f]\|_{H^{s-1}}+\|\pi_j^\e\p\bD(f_0)[f][\beta_0]\|_{H^{s-1}}
\leq C\|\pi_j^\e f\|_{H^s}+ K\|f\|_{H^{s'}},
\end{align}
and similarly, using \eqref{DFB} and \eqref{normB} with $s$ replaced by $s'$, 
\begin{align}\label{COM11}
\|\bD(f_0)[\pi_j^\e\kB(\tau)[f]]-\pi_j^\e\bD(f_0)[\kB(\tau)[f]]\|_{H^{s-1}}\leq K\|\kB(\tau )[ f]\|_{2}\leq  K\|f\|_{H^{s'}}.
\end{align}
The estimate~\eqref{COM} follows now from~\eqref{COM10}--\eqref{COM11} and Theorem~\ref{T:isom2}.\medskip

\noindent{\em Step 3.} Given $\tau\in[0,1]$ and $ -N+1\leq j\leq N$,  let $\bB_{j,\tau}\in\kL(H^{s}(\R)^2, H^{s-1}(\R)^2)$ denote the Fourier multipliers
\begin{align*}
\bB_{j,\tau}:=\frac{a_\mu\sigma}{2\pi}
\begin{pmatrix}
a_1(\tau f_0)(\xi_j^\e)B_{0,0}\circ(d/d\xi)\\[1ex]
-a_2(\tau f_0)(\xi_j^\e)B_{0,0}\circ(d/d\xi)
\end{pmatrix},\quad |j|<N,
\text{ and }
\bB_{N,\tau}:=\frac{a_\mu\sigma}{2\pi}
\begin{pmatrix}
0\\[1ex]
-B_{0,0}\circ(d/d\xi)
\end{pmatrix}.
\end{align*}
We next prove that given $\nu>0$,   we have
\begin{align}\label{ESt2} 
\|\pi_j^\e \kB(\tau)[f]-\bB_{j,\tau}[\pi_j^\e f]\|_{H^{s-1}}\leq \nu\|\pi_j^\e f\|_{H^{s}}+K\|f\|_{H^{s'}}
\end{align}
for all $  -N+1\leq j\leq N$, $\tau\in[0,1],$    $f\in H^s(\mathbb{R})$  and all   sufficiently small $\e$.
To start, we multiply \eqref{betatau} by $\pi_j^\e$  and get
 \begin{align}\label{betatau'}
\pi_j^\e\kB(\tau)[f]
=2a_\mu\pi_j^\e\big[\p G(\tau f_0)[f]-\tau\big(\bD(f_0)[\kB(\tau)[f]]
+\p\bD(f_0)[f][\beta_0]\big)\big]
\end{align}
We  consider the terms on the right hand side of~\eqref{betatau'} one by one. To deal with the first term we recall \eqref{pPsi1}-\eqref{pPsi3}.
 Repeated use of Lemma~\ref{L:AL2} and Lemma~\ref{L:AL3} then shows that 
 \begin{equation}\label{T2n-1}
\|2a_\mu\pi_j^\e\p G(\tau f_0)[f]-\bB_{j,\tau}[\pi_j^\e f]\|_{H^{s-1}}\leq \frac{\nu}{3}\|\pi_j^\e f\|_{H^{s}}+K\|f\|_{H^{s'}}
\end{equation}
for $|j|\leq N-1$,  while  Lemma~\ref{L:AL4}, Lemma~\ref{L:AL5}, and Lemma~\ref{L:AL6}  yield
\begin{equation}\label{T2n}
\|2a_\mu\pi_N^\e\p G(\tau f_0)[f]-\bB_{N,\tau}[\pi_N^\e f]\|_{H^{s-1}}\leq  \frac{\nu}{3}\|\pi_N^\e f\|_{H^{s}}+K\|f\|_{H^{s'}}
\end{equation}
provided that $\e$ is sufficiently small.

We estimate the second term on the right of \eqref{betatau'} and let $|j|\leq N-1$ first.
Combining~\eqref{DFB},  Lemma  \ref{L:AL2}, Lemma~\ref{L:AL3},  \eqref{normB}  with $s$ replaced by $s'$, and \eqref{COM}     we obtain
\begin{equation}\label{T1n-1}
\begin{aligned}
&\|\pi_j^\e\bD(f_0)[\kB(\tau)[f]]\|_{H^{s-1}}\\[1ex]
&\leq\Bigg\|\pi_j^\e\begin{pmatrix}
B_{0,2}^0&B_{1,2}^0\\[1ex]
B_{1,2}^0&B_{2,2}^0
\end{pmatrix}
\begin{pmatrix}
  f_0'\kB_1(\tau)[f]\\[1ex]
  f_0'\kB_2(\tau)[f]
\end{pmatrix}\\[1ex]
&\hspace{1cm}-\frac{f_0'(\xi_j^\e)}{(1+f_0'^2(\xi_j^\e))^2}
\begin{pmatrix}
1&f_0'(\xi_j^\e)\\[1ex]
f_0'(\xi_j^\e)&f_0'^2(\xi_j^\e)
\end{pmatrix}
\begin{pmatrix}
B_{0,0}[\pi_j^\e \kB_1(\tau)[f]]\\[1ex]
B_{0,0}[\pi_j^\e \kB_2(\tau)[f]]
\end{pmatrix}\Bigg\|_{H^{s-1}}\\[2ex]
&\hspace{0,45cm}+\Bigg\|\pi_j^\e\begin{pmatrix}
B_{1,2}^0&B_{2,2}^0\\[1ex]
B_{2,2}^0&B_{3,2}^0
\end{pmatrix}
\begin{pmatrix}
\kB_1(\tau)[f]\\[1ex]
\kB_2(\tau)[f]
\end{pmatrix}\\[1ex]
&\hspace{1.45cm}-\frac{f_0'(\xi_j^\e)}{(1+f_0'^2(\xi_j^\e))^2}
\begin{pmatrix}
1&f_0'(\xi_j^\e)\\[1ex]
f_0'(\xi_j^\e)&f_0'^2(\xi_j^\e)
\end{pmatrix}
\begin{pmatrix}
B_{0,0}[\pi_j^\e \kB_1(\tau)[f]]\\[1ex]
B_{0,0}[\pi_j^\e \kB_2(\tau)[f]]
\end{pmatrix}\Bigg\|_{H^{s-1}}\\[1ex]
&\leq  \frac{\nu}{6|a_\mu|}\|\pi_j^\e f\|_{H^{s}}+K\|f\|_{H^{s'}}
\end{aligned}
\end{equation}
provided that $\e$ is sufficiently small.
Similarly, if $j=N$, then Lemma~\ref{L:AL5}, Lemma~\ref{L:AL6}, \eqref{normB}  with $s$ replaced by $s'$, and \eqref{COM}  imply that
\begin{equation}\label{T1n}
\|\pi_N^\e\bD(f_0)[\kB(\tau)[f]]\|_{H^{s-1}}\leq  \frac{\nu}{6|a_ \mu|}\|\pi_N^\e f\|_{H^{s}}+K\|f\|_{H^{s'}}
\end{equation} 
provided that $\e$ is sufficiently small.

 It remains to consider the term $\pi_j^\e\p\bD(f_0)[f][\beta_0]$ on the right of~\eqref{betatau'}. To this end we argue similarly as in the proof of \eqref{T1n-1} by adding and 
 subtracting suitable localization operators.
Recalling \eqref{pDE1}-\eqref{pDE3}, we get from Lemma~\ref{L:AL2} and Lemma~\ref{L:AL3} if $|j|\leq N-1$, respectively from   Lemma~\ref{L:AL4} and Lemma~\ref{L:AL6} if~${j=N}$, that 
\begin{equation}\label{T3n-1}
\|\pi_j^\e\p\bD(f_0)[f][\beta_0]\|_{H^{s-1}}\leq \frac{\nu}{6|a_\mu|}\|\pi_j^\e f\|_{H^{s}}+K\|f\|_{H^{s'}}
\end{equation}
provided that $\e$ is sufficiently small.
The estimate \eqref{ESt2} follows now from \eqref{betatau'}-\eqref{T3n-1}.\medskip

\noindent{\em Step 4.} We are now in a position to  localize the operators~$\big[f\mapsto\kB_2(\tau)[f]-\tau f_0'\kB_1(\tau)[f]\big]$.
The estimate \eqref{ESt2} shows that, choosing $\e$ sufficiently small, we have
\begin{align}\label{ESt3an-1}
\Big\|\pi_j^\e \kB_2(\tau)[f]+\frac{a_\mu\sigma}{2\pi}a_2(\tau f_0)(\xi_j^\e)B_{0,0}[(\pi_j^\e f)']\Big\|_{H^{s-1}}\leq  \frac{\mu|\mu_+-\mu_-|}{3}\|\pi_j^\e f\|_{H^{s}}+K\|f\|_{H^{s'}}
\end{align}
for $|j|\leq N-1$ and
\begin{align}\label{ESt3an}
\Big\|\pi_N^\e (\kB_2(\tau)[f]+\frac{a_\mu\sigma}{2\pi}B_{0,0}[(\pi_N^\e f)']\Big\|_{H^{s-1}}\leq  \frac{\mu|\mu_+-\mu_-|}{3}\|\pi_N^\e f\|_{H^{s}}+K\|f\|_{H^{s'}}.
\end{align}

Moreover, for $|j|\leq N-1$, we write in view of $\chi_j^\e\pi_j^\e=\pi_j^\e$
\begin{equation*}
\begin{aligned}
&\hspace{-0.5cm}\Big\|  \pi_j^\e f_0'\kB_1(\tau)[f]-\frac{a_\mu\sigma}{2\pi}f_0'(\xi_j^\e)a_1(\tau f_0)(\xi_j^\e)B_{0,0}[(\pi_j^\e f)']\Big\|_{H^{s-1}}\\[1ex]
&\leq \|   \chi_j^\e (f_0'-f_0'(\xi_j^\e))\pi_j^\e \kB_1(\tau)[f]\|_{H^{s-1}}\\[1ex]
&\hspace{0.45cm}+C\Big\|\pi^\e_j\kB_1(\tau)[f]-\frac{a_\mu\sigma}{2\pi}a_1(\tau f_0)(\xi_j^\e)B_{0,0}[(\pi_j^\e f)']\Big\|_{H^{s-1}}.
\end{aligned}
\end{equation*}
The first term on the right hand side may be estimated by using \eqref{normB} (with~$s$ replaced by~$s'$),~\eqref{MES}, \eqref{COM}, and the fact that  $f_0'\in{\rm C}^{s-3/2}(\R)$.
For the second term  we rely on~\eqref{ESt2}. 
Hence, if $\e$ is sufficiently small then
\begin{equation}\label{ESt3bn-1}
\begin{aligned}
&\hspace{-0.5cm}\Big\|  \pi_j^\e f_0'\kB_1(\tau)[f]-\frac{a_\mu\sigma}{2\pi}f_0'(\xi_j^\e)a_1(\tau f_0)(\xi_j^\e)B_{0,0}[(\pi_j^\e f)']\Big\|_{H^{s-1}} \\[1ex]
&\leq \frac{\mu|\mu_+-\mu_-|}{3}\|\pi_j^\e f\|_{H^{s}}+K\|f\|_{H^{s'}}.
\end{aligned}
\end{equation}
For $j=N$, it follows from \eqref{normB} (with~$s$ replaced by $s'$), \eqref{MES}, \eqref{COM}, and the fact that~$f_0' $ vanishes at infinity that 
\begin{equation}\label{ESt3bn}
\| \pi_N^\e f_0'\kB_1(\tau)[f]\|_{H^{s-1}} \leq \frac{\mu|\mu_+-\mu_-|}{3}\|\pi_N^\e f\|_{H^{s}}+K\|f\|_{H^{s'}}.
\end{equation}
The desired claim \eqref{D1} follows now from \eqref{PSI}, \eqref{aa1}, \eqref{ESt3an-1}, and \eqref{ESt3bn-1} if $|j|\leq N-1$, respectively from~\eqref{PSI}, \eqref{aa1}, \eqref{ESt3an}, and \eqref{ESt3bn} if $j=N.$
\end{proof}

We now investigate the Fourier multipliers $\bA_{j,\tau} $ found in Theorem \ref{T:AP}. We recall  the definitions \eqref{defa12}, \eqref{defalphabeta}, and \eqref{defAjtau} and observe that as 
$f_0', \, \beta_0^1, a_i(\tau f_0)\in H^{s-1}(\mathbb{R})$, $i=1,\,2$ and~$\tau\in[0,1]$, there is a constant $\eta\in(0,1)$ such that 
\[
\eta\leq \alpha_\tau\leq \frac{1}{\eta}\quad\text{and}\quad |\beta_\tau|\leq \frac{1}{\eta},\qquad\tau\in[0,1].
\]
Based on this, it can be shown as in  \cite[Proposition 4.3]{MBV19}, that there is 
a constant~${\kappa_0\geq 1}$   such that for all $\e\in(0,1)$, $ -N+1\leq j\leq N$, and $\tau\in[0,1]$ we have
 \begin{align}
\bullet &\quad \text{$\lambda-\bA_{j,\tau}\in \kL(H^s(\mathbb{R}),H^{s-1}(\mathbb{R}))$ is an isomorphism for all $\re\lambda\geq 1,$}\label{L:FM1}\\[1ex]
\bullet &\quad  \kappa_0\|(\lambda-\bA_{j,\tau})[f]\|_{H^{s-1}}\geq |\lambda|\cdot\|f\|_{H^{s-1} }+\|f\|_{H^s}, \qquad f\in H^{s}(\mathbb{R}),\, \re\lambda\geq 1\label{L:FM2}.
\end{align}

The  properties \eqref{L:FM1}-\eqref{L:FM2} together  with Theorem~\ref{T:AP}  enable us to prove Theorem~\ref{T:GP}.

\begin{proof}[Proof of Theorem \ref{T:GP}]
Let $s'\in(3/2,s)$ and  let $\kappa_0\geq1$ be the constant   in~\eqref{L:FM2}. 
  Theorem~\ref{T:AP} with $\mu:=1/2\kappa_0 $ implies that there are $\e\in(0,1)$, a constant $K=K(\e)>0$
 and  bounded operators $\bA_{j,\tau}\in\kL(H^s(\mathbb{R}), H^{s-1}(\mathbb{R}))$, $ -N+1\leq j\leq N$ and $\tau\in[0,1],$ satisfying 
 \begin{equation*} 
  2\kappa_0\|\pi_j^\e\Psi(\tau )[f]-\bA_{j,\tau}[\pi^\e_j f]\|_{H^{s-1}}\leq \|\pi_j^\e f\|_{H^{s}}+2\kappa_0 K\|  f\|_{H^{s'}},\qquad f\in H^s(\mathbb{R}).
 \end{equation*}
Moreover,  \eqref{L:FM2}  yields
  \begin{equation*} 
    2\kappa_0\|(\lambda-\bA_{j,\tau})[\pi^\e_jf]\|_{H^{s-1}}\geq 2|\lambda|\cdot\|\pi^\e_jf\|_{H^{s-1}}+ 2\|\pi^\e_j f\|_{H^s}
 \end{equation*}
 for all $-N+1\leq j\leq N$, $\tau\in[0,1],$  $\re \lambda\geq 1$, and  $f\in H^s(\mathbb{R})$.
The latter  estimates combined lead us to
 \begin{align*}
   2\kappa_0\|\pi_j^\e(\lambda-\Psi(\tau ))[f]\|_{H^{s-1}}
   \geq& 2\kappa_0\|(\lambda-\bA_{j,\tau})[\pi^\e_jf]\|_{H^{s-1}}-2\kappa_0\|\pi_j^\e\Psi(\tau)[f]-\bA_{j,\tau}[\pi^\e_j f]\|_{H^{s-1}}\\[1ex]
   \geq& 2|\lambda|\cdot\|\pi^\e_j f\|_{H^{s-1}}+ \|\pi^\e_j f\|_{H^s}-2\kappa_0K\|  f\|_{H^{s'}}.
 \end{align*}
Summing  over $j$, we deduce from \eqref{EQNO}, Young's inequality, and the interpolation property~\eqref{IP}  that
 there exist constants  $\kappa\geq1$  and~$\omega>1 $ such that 
  \begin{align}\label{KDED}
   \kappa\|(\lambda-\Psi(\tau ))[f]\|_{H^{s-1}}\geq |\lambda|\cdot\|f\|_{H^{s-1}}+ \| f\|_{H^s}
 \end{align}
for all   $\tau\in[0,1],$   $\re \lambda\geq \omega$, and  $f\in H^s(\mathbb{R})$.

From \eqref{FMPsi0} we also deduce that  $\omega-\Psi(0) \in \kL(H^s(\mathbb{R}), H^{s-1}(\mathbb{R}))$ is an isomorphism.
This together with method of continuity \cite[Proposition I.1.1.1]{Am95} and  \eqref{KDED}  implies  that also 
\begin{align}\label{DEDK2}
   \omega-\Psi(1)=\omega-\p\Phi(f_0)\in\kL(H^s(\mathbb{R}), H^{s-1}(\mathbb{R})) 
 \end{align}
 is an isomorphism.
The estimate \eqref{KDED} (with $\tau=1$) and \eqref{DEDK2}  finally imply that  $\p\Phi(f_0)$ generates an analytic semigroup 
in $\kL(H^{s-1}(\R))$, cf. \cite[Chapter~I]{Am95}, and the proof is complete.
\end{proof}

  We are now in a position to prove the main result, for which we can exploit abstract theory for fully nonlinear parabolic problems from \cite{L95}.

 \begin{proof}[Proof of Theorem~\ref{MT1}]
{\em  Well-posedness:} Given $\alpha\in(0,1)$,   $T>0$, and a Banach space $X$  we set
 \begin{align*}
 {\rm C}^{\alpha}_{\alpha}((0,T], X):=\Big\{f:(0,T]\longrightarrow X\,\Big|\,\|f\|_{C_\alpha^\alpha}:=\sup_t\|f(t)\|+\sup_{s\neq t}\frac{\|t^\alpha f(t)-s^\alpha f(s)\|}{|t-s|^\alpha}<\infty\Big\}.
 \end{align*}
The property \eqref{REGphi} together with  Theorem \ref{T:GP} shows  that the assumptions of \cite[Theorem~8.1.1]{L95} 
are  satisfied for the evolution problem~\eqref{NNEP}.
According to this theorem, \eqref{NNEP} has, for each~${f_0\in H^{s}(\mathbb{R})}$, a local solution $f(\cdot;f_0)$ such that
\[ f\in {\rm C}([0,T],H^{s}(\mathbb{R}))\cap {\rm C}^1([0,T], H^{s-1}(\mathbb{R}))\cap {\rm C}^{\alpha}_{\alpha}((0,T], H^s(\mathbb{R})),\] 
where   $T=T(f_0)>0$ and   $\alpha\in(0,1)$ is fixed (but arbitrary).
This solution is unique within the set
\[
  \bigcup_{\alpha\in(0,1)}{\rm C}^{\alpha}_{\alpha}((0,T],H^s(\mathbb{R})) \cap {\rm C}([0,T],H^{s}(\mathbb{R}))\cap {\rm C}^1([0,T], H^{s-1}(\mathbb{R})).
 \]
We improve the uniqueness property by showing that  the solution is unique within 
$${\rm C}([0,T],H^{s}(\mathbb{R}))\cap {\rm C}^1([0,T], H^{s-1}(\mathbb{R})).$$
Indeed,  let $f$ now be any solution to \eqref{NNEP} in that space, let $s'\in (3/2,s)$ be fixed and set~${\alpha:= s-s'\in(0,1)}$. 
 Using  \eqref{IP}, we find a constant~$C>0$ such that
 \begin{equation}\label{BO}
\|f(t_1)-f(t_2)\|_{H^{s'}}   \leq C|t_1-t_2|^\alpha,\qquad t_1,\, t_2\in[0, T],
 \end{equation}
which shows in particular that $f \in   {\rm C}^{\alpha}_{\alpha}((0,T], H^{s'}(\mathbb{R}))$.
The uniqueness statement of\cite[Theorem 8.1.1]{L95}  applied in the context of~\eqref{NNEP} with 
$\Phi\in {\rm C}^{\infty}(H^{s'}(\mathbb{R}), H^{s'-1}(\mathbb{R}))$
 establishes the uniqueness claim.
This unique solution can be extended up to a maximal existence time~${T_+(f_0)}$, see \cite[Section 8.2]{L95}. 
Finally, \cite[Proposition 8.2.3]{L95} shows that the solution map defines  a semiflow on $H^s(\mathbb{R})$ which, according to \cite[Corollary 8.3.8]{L95}, is smooth in the open set~$\{(t,f_0)\,|\, 0<t<T_+(f_0)\}$.
This proves~(i). \medskip

\noindent{\em  Parabolic smoothing:} The uniqueness result established in (i) enables us  to 
use a parameter trick   applied also to other problems, cf., e.g., \cite{An90, ES96, PSS15, MBV19}, in order to establish~(iia) and~(iib).
The proof details  are similar to those in \cite[Theorem~1.2~(v)]{MBV18} or \cite[Theorem 2~(ii)]{AM21x} and therefore we omit them.  \medskip

\noindent{\em  Global existence:} We prove the statement by contradiction. 
Assume  there exists a maximal solution   
$f\in {\rm C}([0,T_+),H^{s}(\mathbb{R}))\cap {\rm C}^1([0,T_+), H^{s-1}(\mathbb{R}))$ to \eqref{NNEP} with $T_+<\infty$ and such that  
\begin{align}\label{BOUN1}
\sup_{[0,T_+)} \|f(t)\|_{H^s}<\infty.
\end{align}
Recalling that $\Phi$   maps bounded sets in~$H^s(\R)$ to bounded sets in~$H^{s-1}(\R) $, we get
 \begin{align}\label{BOUN2}
\sup_{t\in [0,T_+)}\Big\| \frac{df}{dt}(t)\Big\|_{H^{s-1}}=\sup_{t\in[0,T_+)}\|\Phi(f(t))\|_{H^{s-1}}<\infty.
\end{align} 
Let $s'\in(3/2,s)$ be fixed. 
Arguing as above, see \eqref{BO}, from the bounds \eqref{BOUN1} and~\eqref{BOUN2} we get that ${f:[0,T_+)\longrightarrow H^{s'}(\mathbb{R})}$ is uniformly continuous. 
Applying \cite[Theorem 8.1.1]{L95} to~\eqref{NNEP} with $\Phi\in {\rm C}^\infty(H^{s'}(\mathbb{R}), H^{s'-1}(\mathbb{R}))$, we may extend the solution $f$   
to  a  time  interval~$[0,T_+')$ with~${T_+<T_+'}$ and such that 
 \[
   f\in {\rm C}([0,T_+'),H^{s'}(\mathbb{R}))\cap {\rm C}^1([0,T_+'), H^{s'-1}(\mathbb{R})).
 \]  
 Since by (iib) (with $s$ replaced by $s'$)  we have  $f\in {\rm C}^1((0,T_+'),H^{s}(\mathbb{R}))$, this contradicts  the maximality property of $f$ and the proof is complete.
  \end{proof}

 \appendix
 
 \section{\label{neargamma}The hydrodynamic double-layer potential near $\Gamma$}
 
Given  $f\in H^3(\RRM)$ and $\beta\in H^2(\RRM)$, we  let  $(w,q)$  be given by \eqref{defw1} and \eqref{defq1}.  
 We recall the definitions \eqref{defD} of $\bD(f)$ and \eqref{defB0} of the operators $B^0_{n,m}$.
 \begin{lemma}\label{l:neargamma}
 We have $w^\pm\in {\rm C}^1(\ov{\Omega^\pm},\RRM^2)$, $q^\pm\in {\rm C}(\ov{\Omega^\pm})$, $w^\pm|_\Gamma\circ\Xi\in H^2(\RRM)^2$, and
 \be\label{jumpwq}
 \left.\begin{array}{rcll}
 w^\pm&=&\displaystyle\Big(-\bD(f)[\beta]\pm\frac{1}{2}\beta\Big)\circ\Xi^{-1}&
 \quad\mbox{\rm on $\Gamma$,}\\[1ex]{}
 [T_1(w,q)](\nu\circ\Xi^{-1})&=&0&\quad\mbox{\rm on $\Gamma$.}
 \end{array}\right\}
 \ee
 \end{lemma}
 \begin{proof}
 For $j=0,\ldots,3$, let $\clz_j\in {\rm C}^1(\RRM^2\setminus\{0\})$ be given by
 \[\clz_j(y):= \frac{y_1^{3-j}y_2^j}{|y|^4},\qquad y\in \RRM^2\setminus\{0\}.\]
 Given $\phi\in H^1(\RRM)$, we define the function $Z_j[\phi]:\R^2\setminus\Gamma\longrightarrow\R$, $j=0,\ldots,3,$  by
 \[
 Z_j[\phi](x):= \int_\RRM\clz_j(r)\phi\,ds,  \quad x\in\RRM^2\setminus\Gamma,\, r:=x-(s,f(s)).
 \]
  Recalling $\eqref{defw2}_1$, we have
 \[w=\frac{1}{\pi}\left(
 -\left(\begin{array}{cc}
 Z_0&Z_1\\
 Z_1&Z_2
 \end{array}\right)[f'\beta]
 +\left(\begin{array}{cc}
 Z_1&Z_2\\
 Z_2&Z_3
 \end{array}\right)[\beta]\right).\]
 
 It is shown in \cite[Lemma A.1]{MP2021} that  $Z_j[\phi]^\pm \in{\rm C}(\ov{\0^\pm})$, with 
 \be\label{jumpZ}
 \begin{pmatrix}
 Z_0[\phi]^\pm\\
 Z_1[\phi]^\pm\\
 Z_2[\phi]^\pm\\
 Z_3[\phi]^\pm
 \end{pmatrix}\circ\Xi=\begin{pmatrix}
 B^0_{ 0,2}(f)[\phi]\\
 B^0_{ 1,2}(f)[\phi]\\
 B^0_{ 2,2}(f)[\phi]\\
 B^0_{ 3,2}(f)[\phi]
 \end{pmatrix}
 \mp\frac{\pi}{2\omega^4}
 \begin{pmatrix}
 {f'}^3+3f'\\
 {f'}^2-1\\
 {f'}^3-f'\\
 -3{f'}^2-1
 \end{pmatrix}\phi.
 \ee
Consequently,  $w^\pm\in {\rm C}(\ov{\Omega^\pm},\RRM^2)$, and the jump relations
 \eqref{jumpZ} imply \eqref{jumpwq}$_1$. 
  Moreover, recalling Corollary~\ref{C:1}, we get $w^\pm|_{\G}\circ\Xi\in H^2(\R)^2$.
Further,~$q^\pm\in{\rm C}(\ov{\0^\pm})$ follows from \cite[Lemma~2.1]{MBV18}.
 
 Exchanging integration with respect to $s$ and differentiation with respect to $x$ by dominated convergence we find from \eqref{defT}, \eqref{defw1}, and \eqref{defq1} that
 \begin{align}
     \p_lw_j(x)&=\int_\RRM\p_l\clw_j^{i,k}(r)\nu_i\beta_k\omega\,ds,
     \label{wtrep1}\\
     (T_1(w,q))_{jl}(x)&=\int_\RRM(-\delta_{jl}\clq^{i,k}+\p_l\clw_j^{i,k}+\p_j\clw_j^{i,k}
     )(r)\nu_i\beta_k\omega\,ds
     \label{wtrep2}
 \end{align}
 for $x\in\RRM^2\setminus\Gamma$ and $  l,\, j=1,\,  2$.
 
 For $E\subset\RRM^2$ open, $\clz\in {\rm C}^1(E)$, $i=1,2$, we define $\rot \clz:= (\rot^1 \clz, \rot^2 \clz)\in {\rm C}(E,\RRM^2)$ by
 \[\rot^i\clz:=
 \left\{\begin{array}{rl}
 -\p_2 \clz&\text{ if $i=1$,}\\
 \p_1\clz&\text{ if $i=2$}.\end{array}\right.\]
 With this notation, we find from integration by parts
 \[\int_\RRM (\rot^i\clz_j)(r)
 \nu_i\phi\omega\,ds=\int_\RRM(f'\p_2\clz_j(r)+\p_1\clz_j(r))\phi\,ds
 =-\int_\RRM\p_s(\clz_j(r))\phi\,ds= Z_j[\phi'].\]
 Together with \eqref{wtrep1}, \eqref{wtrep2}, and the identities
 \begin{align*}
      \p_1\clw_1^{i,1}&=-\p_2\clw_2^{i,1}=-\p_2\clw_1^{i,2}= \frac{1}{\pi}\rot^i\clz_1,\\[1ex]
      \p_1\clw_1^{i,2}&=-\p_2\clw_2^{i,2}= \p_1\clw_2^{i,1}=\frac{1}{\pi}\rot^i\clz_2,\\[1ex]
      \p_2\clw_1^{i,1}&=-\frac{1}{\pi}\rot^i\clz_0,\\[1ex]
      \p_1\clw_2^{i,2}&= \frac{1}{\pi}\rot^i\clz_3
 \end{align*}
 and
 \begin{align*}
      \clq^{i,1}&= \frac{1}{\pi}\rot^i(-\clz_1-\clz_3),\\[1ex]
      \clq^{i,2}&= \frac{1}{\pi}\rot^i(\clz_0+\clz_2),   
 \end{align*}
 this yields $w^\pm\in {\rm C}^1(\ov{\Omega^\pm},\RRM^2)$ and \eqref{jumpwq}$_2$.

 \end{proof}

{\bf Acknowledgements:} The research leading to this paper was carried out while the second author enjoyed the hospitality of DFG Research Training Group 2339
``Interfaces, Complex Structures, and Singular Limits in Continuum Mechanics - Analysis and Numerics'' at the Faculty of Mathematics of Regensburg University.

 \bibliographystyle{siam}
 \bibliography{MP}
\end{document}